\documentclass[a4paper,oneside,reqno,12pt]{amsart}
\usepackage[utf8]{inputenc}
\usepackage{comment}
\usepackage{enumerate}
\usepackage{amsmath,amssymb,amsthm}
\usepackage{paralist}
\usepackage[dvipsnames]{xcolor}
\usepackage{graphicx}  
\usepackage[colorlinks=true]{hyperref}
\hypersetup{urlcolor=blue, citecolor=red}
\usepackage{mathtools}
\newtheorem{theorem}{Theorem}[section]
\newtheorem{example}{Example}[section]

\newtheorem{lemma}[theorem]{Lemma}
\theoremstyle{definition}

\newcommand{\rr}{\mathbf{R}}
\newcommand{\ttt}{\mathbf{T}}
\newcommand{\eff}{{\mathrm{eff}}}

\newcommand{\ve}{\varepsilon}

\newcommand{\ii}[1]{\int_{#1}}
\newcommand{\Oe}{\Omega_\varepsilon}
\newcommand{\Se}{\Sigma_\varepsilon}
\newcommand{\Ge}{\Gamma_\varepsilon^\pm}

\newcommand{\mop}[1]{{\mathop{\mathrm{#1}}}}
\title[Localization of eigenfunctions]{Localization of eigenfunctions in a thin domain with locally periodic oscillating boundary}

\author{K. Pettersson}
\address{Chalmers University of Technology, Sweden}
\email{klaspe@chalmers.se}

\date{\today}


\thanks{}

\makeatletter
\g@addto@macro{\endabstract}{\@setabstract}
\newcommand{\authorfootnotes}{\renewcommand\thefootnote{\@fnsymbol\c@footnote}}%
\makeatother

\begin{document}
\maketitle

%
%
%

\begin{abstract}
We study a Dirichlet spectral problem for a second-order elliptic operator with locally periodic coefficients in a thin domain. The boundary of the domain is assumed to be locally periodic.
When the thickness of the domain $\ve$ tends to zero, the eigenvalues are of order $\ve^{-2}$ and described in terms of the first eigenvalue $\mu(x_1)$ of an auxiliary spectral cell problem parametrized by $x_1$, 
while the eigenfunctions localize with rate $\sqrt\ve$.
\end{abstract}


\section{Introduction}\label{sec:introduction}


This paper deals with the leading terms in the asymptotics
of  the eigenvalues and the eigenfunctions to the
following Dirichlet spectral problem in a thin domain with a periodically oscillating boundary:
\begin{align}\label{eq:intro}
-\mop{div}\big( A\big(x_1, \frac{x}{\ve}\big) \nabla u^\ve \big)  & = \lambda^\ve u^\ve \quad \text{ in } \Omega_\ve, 
\end{align}
with the Dirichlet condition on the boundary $\partial \Omega^\ve$.
The asymptotics is established in the case of locally periodically oscillating coefficients, under structural assumptions given in terms of an eigenvalue problem on the periodicity cell.

In~\cite{friedlander2009spectrum}, Friedlander and Solomyak studied the 
spectrum of the Dirichlet Laplacian in a narrow strip
\begin{align*}
\Omega^\ve & = \{ (x,y) \in \mathbf{R}^2 : -a < x < a, \, 0 < y < \ve h(x) \},
\end{align*}
where $\ve$ is a small positive parameter.
The following structural assumptions were made:
$x = 0$ is the unique global maximum of positive function $h(x)$,
such that
\begin{align*}
h(x) & > 0 \text{ continuous on } I = [-a,a],\, a > 0 \\
h(x) & \text{ is } C^1 \text{ on } I \setminus \{ 0 \}\text{, and }\\
h(x) & = \begin{cases}
M - c_+ x^m + O(x^{m+1}), & x > 0, \\
M - c_- |x|^m + O(x^{m+1}), & x < 0,
\end{cases}\\
&  M, c_-, c_+ > 0, \quad m \ge 1.
\end{align*}
In particular, the authors have proved the following asymptotics result.

\begin{theorem}[Friedlander, Solomyak, 2009]\label{tm:salmiak}
Let $\lambda^\ve_j$ be the eigenvalues of $-\Delta$ on $\Omega^\ve$
with Dirichlet condition on $\Omega^\ve$.
Then 
\begin{align*}
\mu_j & = \lim_{\ve \to 0} \ve^{4(m+2)^{-1}} \Big( \lambda_j^\ve - \frac{\pi^2}{M^2 \ve^2} \Big),
\end{align*}
where $\mu_j$ are the eigenvalues of the operator on $L^2(\mathbf{R})$ given by \begin{align*}
-\frac{d^2}{dx^2} + q(x), \quad 
q(x) = 
\begin{cases}
2\pi M^{-3}c_+ x^m, & x > 0, \\
2\pi M^{-3}c_- x^m, & x < 0.
\end{cases}
\end{align*}
\end{theorem}
One notes that Theorem~\ref{tm:salmiak} tells that the decay of the profile (or diameter) $h(x_1)$ in the vicinity of its maximum dictates the growth of the eigenvalues $\lambda^\ve_j$,
and the rate of localization of the eigenfunctions,
as $\ve$ tends to zero.

The present paper aims at understanding the
effect of oscillating
coefficients and oscillating boundary on the spectral asymptotics.
In Theorem~\ref{tm:salmiak}, the authors imposed structural assumptions on the strip profile
$h(x_1)$, requiring for $h(x_1)$ a unique maximum point. In singularly perturbed problems, when the leading term of the asymptotics contains an oscillating function, being the first eigenfunction of an auxiliary spectral cell problem, a standard assumption is on the corresponding first eigenvalue (see for example \cite{AlPi-2002}). In our case, we use also the factorization technique, impose an assumption on the first 
eigenvalue $\mu_1(x_1)$ in the periodicity cell eigenvalue problem (see (\ref{eq:H})), and require that it has a unique minimum point. Even if such assumptions are standard,
it is sometimes not apparent how they are related to the geometry of the domain or the coefficients of the equation.
In the present paper we show that the assumption in Theorem~\ref{tm:salmiak} is a specific case of (\ref{eq:H}) (see Example~\ref{ex:hcase}).
We show that the global minimizer of $\mu_1$ determines the point where the eigenfunctions localize. In addition, the growth
of $\mu$ in the vicinity of its minimum
dictates the rate of localization as well as the behavior of
the eigenvalues.

The method we use in order to obtain the leading terms of 
the asymptotics of the eigenvalue and the eigenfunctions
to~\eqref{eq:intro} is homogenization via two-scale convergence.
In particular, we use the singular measure approach,
and the two-scale compactness theorem proved in~\cite{zhikov2000extension,Zh-2004} (see also \cite{BouFra-01}),
which is expressed in the current setting in~\cite{pettersson2017two}.

The fundamentals of spectral asymptotics for elliptic operators
has been considered in~\cite{vishik1957regular}, which has been
successfully applied in homogenization problems (see for example~\cite{oleinik2009mathematical}).
There are many works in which localization of eigenfunctions
have been identified in this context or due to oscillating coefficients
and thin domains.
We believe the works most closely related to our problem are the following.
In \cite{piatnitski1998asymptotic,allaire2002uniform}, spectral asymptotics and localization of eigenvalues in bulk domains with large potentials were considered. 
In \cite{pankratova2015spectral}, spectral asymptotics and localization of eigenvalues in thin domains with large potentials were considered.
Homogenization in domains with low amplitude oscillating boundaries
were considered in \cite{MePo-2010, ArPe-2011, pettersson2017two}. Spectral asymptotics of the Laplace operator in thin domains with slowly varying thickness were considered in \cite{FrSo-2009}, \cite{borisov2010asymptotics}, \cite{NaPeTa-2016}, where under the Dirichlet boundary conditions the localization of eigenfunctions occur.

In order to capture the oscillations of the eigenfunctions, in \cite{piatnitski1998asymptotic,allaire2002uniform} the factorization by an eigenfunction of an auxiliary spectral cell problem was used. In the present paper, we also use the factorization technique. However, due to the homogeneous Dirichlet boundary conditions, the first eigenfunction $\psi_1(x_1, y)$ of the auxiliary spectral problem (\ref{eq:auxproblem}) vanishes on the boundary of the cell, and the new unknown function satisfies a problem with degenerate coefficients
posed in a weighted Sobolev space.

The rest of this paper is organized as follows.
In Section~\ref{sec:problem}, we describe the domain assumptions and
state the problem and the hypotheses.
In Section~\ref{sec:asymptotics}, we establish a priori estimates 
for the eigenvalues and the eigenfunctions,
and prove the spectral asymptotics result of this paper.
In Section~\ref{sec:computing}, we describe a scheme to compute the 
 the leading terms in the expansions of the eigenvalue and the eigenfunctions
 to~\eqref{eq:intro}.

\section{Problem statement}\label{sec:problem}

The problem we consider is to describe the leading terms 
in the asymptotics of the first eigenvalue $\lambda_1^\ve$ and eigenfunction $u_1^\ve$ to the problem~\eqref{eq:originalproblem}, as $\ve$ tends to zero.
The domain under consideration is a thin cylinder with a periodically 
oscillating boundary, and the coefficients are assumed to be smooth, periodically oscillating,
and to satisfy the strong ellipticity condition.
The Dirichlet condition will be emposed on the boundary.
Here we specify the assumptions on the domain, and the coefficients and boundary conditions separately.

\subsection{A thin cylinder with a periodically oscillating boundary}\label{subsec:domain}
Let $\varepsilon>0$ be a small parameter; the points in $\rr^d$ are denoted by $x=(x_1, x')$, and $I = (-\frac{1}{2},\frac{1}{2})$.
We are going to work in a thin cylinder 
\begin{align*}
\Oe & = \Big\{ x=(x_1,x') : x_1 \in I, \, x' \in \ve Q(x_1,\frac{x_1}{\ve}) \Big\}.
\end{align*}
Here $Q(x_1,x_1/\ve)$ describes the locally periodic varying cross section of the cylinder  introduced in the following way.

We denote
\begin{align*}
Q(x_1, y_1) = \{ y' \in \rr^{d-1} : F(x_1,y_1,y')>0 \},
\end{align*}
where $F(x_1,y_1,y')$ is such that
\begin{itemize}
\item[(H1)]
$F(x_1, y_1, y') \in C^{2}(\overline I \times \mathbf T^1 \times \rr^{d-1})$, where $\mathbf T^1$ is the one-dimensional torus.
\item[(H2)]
$Q(x_1, y_1)$ is non-empty, bounded, and simply connected.
\end{itemize}

To ensure that the conditions (H1), (H2) are fulfilled, we can take, for example, $F$ satisfying the assumptions 
\begin{itemize}
\item[(F1)]
For each $x_1$ and $y_1$, $F(x_1, y_1, 0)>0$ and $F(x_1, y_1, y') <0$ for $|y'| \ge R$, for some $R>0$.
This guarantees that $Q(x_1, y_1)$ is not empty and bounded.
\item[(F2)]
$F(x_1, y_1, \cdot)$ does not have a nonpositive local maximum/minimum. This guarantees that $Q(x_1, y_1)$ is simply connected.
\end{itemize}
When $F=F(y_1, y')$ we have a periodic oscillating boundary, when $F=F(x_1, y')$ we are in the case of slowly varying thickness, and finally, when $F=F(y')$ the cylinder has uniform cross-section, constant along the cylinder.

The boundary of $\Oe$ consists of the lateral boundary of the cylinder 
\begin{align*}
\Se = \big\{
x = (x_1,x') : x_1 \in I, \,
F\big(x_1,\frac{x_1}{\ve},\frac{x'}{\ve}\big) = 0
\big\},
\end{align*}
and the bases $\Ge=\big\{\pm \frac{1}{2}\big\} \times \ve Q\big(\pm \frac{1}{2}, \pm \frac{1}{\ve}\big)$.

The periodicity cell depends on $x_1$ and is defined as follows
\begin{align*}
\square(x_1) & = \{y = (y_1,y') \in \mathbf{T}^1 \times \mathbf{R}^{d-1} : y_1 \in \mathbf T^1, \, y' \in Q(x_1,y_1)\} \\
& = \{ y \in \mathbf{T}^1 \times \mathbf{R}^{d-1} : F(x_1, y) > 0 \},
\end{align*}
where $\mathbf T^1$ is the one-dimensional torus, realized with unit measure.
Since $F(x_1,y_1, y')$ is periodic in $y_1$, and $F$ is regular, the boundary of the periodicity cell $\partial \Box(x_1) = \{ y : F(x_1,y) = 0 \}$ is Lipschitz.

\subsection{Dirichlet spectral problem in a thin oscillating domain}\label{subsec:problem}

In the thin domains with oscillating boundary $\Omega_\ve$ of Section~\ref{subsec:domain}, we consider the following Dirichlet spectral problem:
\begin{align}\label{eq:originalproblem}
-\mop{div}\big( A\big(x_1, \frac{x}{\ve}\big) \nabla u^\ve \big)  & = \lambda^\ve u^\ve \quad \text{ in } \Omega_\ve,\\
u^\ve &= 0 \hskip 1cm \text{ on } \partial \Omega_\ve. \nonumber
\end{align}

The matrix $A(x_1,y)$ in \eqref{eq:originalproblem} is assumed to be symmetric with entries $a_{ij}$ in $C^2([0,1], L^\infty(\mathbf{T}^d))$, and satisfy the ellipticity condition
\begin{align*}
A(x_1,y) \xi \cdot \xi \ge \alpha |\xi|^2, \quad x_1 \in I, \, \text{a.e. } y \in \square(x_1), \, \xi \in \rr^d.
\end{align*}

By the Hilbert-Schmidt theorem,
for each $\ve > 0$, the spectrum of \eqref{eq:originalproblem} is discrete and may be arranged as follows:
\begin{align*}
0 < \lambda_1^\ve < \lambda_2^\ve \le \lambda_3^\ve \le \cdots, \qquad \lim_{j\to \infty} \lambda_j^\ve = \infty,
\end{align*}
where the eigenvalues are counted as many times as their finite multiplicity.

Normalized by
\begin{align}
\label{eq:normalization-eigenfunctions}
\int_{\Omega_\ve} u_i^\ve u_j^\ve \, dx = \ve^{1/2} \ve^{d-1} \delta_{ij},
\end{align}
where $\delta_{ij}$ is the Kronecker delta, the corresponding eigenfunctions form an orthogonal basis in $L^2(\Oe)$.

We study the asymptotic behavior of the eigenpairs
$(\lambda^\ve, u^\ve)$, as $\ve$ tends to zero. The homogenization result for the eigenpairs $(\lambda^\ve, u^\ve)$, as $\ve \to 0$, is given in Theorem \ref{tm:one}. In order to present the limit problem, we need an auxiliary spectral cell problem, which is introduced in Section \ref{subsec:estimates} (see (\ref{eq:auxproblem})).

\section{Spectral asymptotics}\label{sec:asymptotics}

In this section, we identify the first two terms in the 
asymptotic expansion of the first eigenvalue $\lambda_1^\ve$ to \eqref{eq:originalproblem}:
\begin{align}\label{eq:a1}
\lambda_1^\ve & = \frac{\mu_0(0)}{\ve^2} + \frac{\lambda_1^0}{\ve} + o\Big(\frac{1}{\ve}\Big), \quad \ve \to 0.
\end{align}
For comparison with the result of Theorem~\ref{tm:salmiak},
for smooth $h$ with quadratic growth at its maximum $(m=2)$, one has for the first eigenvalue $\lambda_1^\ve$ of the Dirichlet Laplacian in $\mathbf{R}^2$, 
\begin{align}\label{eq:a2}
\lambda_1^\ve & = \frac{\pi^2/h(0)^2}{\ve^2} + \frac{\mu_1}{\ve} + o\Big(\frac{1}{\ve}\Big), \quad \ve \to 0,
\end{align}
where $\mu_1$ is the first eigenvalue of a harmonic oscillator.
When comparing the asymptotics in~\eqref{eq:a1} and~\eqref{eq:a2}
one notes that the oscillations in the coefficients and the domain
do not change the structure of the asymptotics,
while the factors in both the leading terms $1/\ve^2$ and $1/\ve$
will change.
In particular, the eigenvalue $\lambda_1^\ve$
is still shifted a bit to the right as the factor $\lambda_1^0$ is positive.
One notes also that the factors $\mu_1(0)$ and $\lambda_1^0$ in (\ref{eq:a1}) agree with $\pi^2/h(0)^2$ and $\mu_1$ in (\ref{eq:a2}),
for the Dirichlet Laplacian in $\mathbf{R}^2$ as explained in Section \ref{sec:computing}.

To obtain the asymptotics~\eqref{eq:a1} we will proceed as follows.
First we will derive a priori estimates for the first eigenvalue $\lambda_1^\ve$ to $\eqref{eq:originalproblem}$.
Then we will make a change of variables with the aid of hypothesis \eqref{eq:H}
such that the eigenvalues of the rescaled problem are bounded.
Finally we show the convergence of the corresponding eigenvalues
and their eigenfunction in $L^2$ using the method of two-scale convergence in domains with measure.
The result, which includes~\eqref{eq:a1}, is stated as Theorem~\ref{tm:one}.

\subsection{A priori estimates for the first eigenvalue $\lambda_1^\ve$}
\label{subsec:estimates}

The radius of the cylinder $\Omega_\varepsilon$ is of order $\ve$,
so the Dirichlet condition on the lateral boundary suggests that in view of the variational principle, the first eigenvalue
$\lambda_1^\ve$ should be of order $\ve^{-2}$, as $\ve$ tends to zero.
Consider a smooth function of the form
\begin{align*}
v^\ve(x) & = v(x_1)w(x'/\ve), \quad 
\end{align*}
$v(x_1) \in C_0^\infty(\mathbf{R})$ with support in $I$,
$w(x') \in C_0^\infty(B_0(\rho))$,
where $B_0(\rho)$ is the ball of a small fixed radius $\rho$ centered at the origin in $\mathbf{R}^{d-1}$
such that $I \times \ve B_0(\rho) \subset \Oe$ (c.f.~(F1)). Using $v^\ve(x)$ as a test function 
in the variational principle for $\lambda_1^\ve$, we obtain
\begin{align*}
    \lambda_1^\ve & = \min \limits_{v \in H_0^1(\Oe)\setminus \{0\}} 
    \frac{\int_{\Oe} A\big(x_1, \frac{x}{\ve}\big)\nabla v \cdot \nabla v\, dx}{\int_{\Oe} v^2 \, dx} 
\le C \,
\frac{\int_{\Oe} |\nabla v^\ve|^2 \, dx}{\int_{\Oe} (v^\ve)^2 \, dx} 
 \le \frac{C}{\ve^2},
\end{align*}
for all small enough $\ve$ such that $v^\ve \in H^1_0(\Oe)$,
and for some constant $C$ independent of $\ve$.
To identify the factor in $1/\ve^2$, and to obtain the
corresponding estimate from below,
one needs to choose the test functions more carefully.
In particular, for $O(1/\ve^2)$ to be sharp an optimal oscillating $\ve$-periodic profile $\psi$ along $x_1$ has to be selected, because $|\nabla \psi(x_1/\ve)|^2 = O(1/\ve^2)$.

Let $(\mu_1, \psi_1(x_1, y))$, for each $x_1 \in \overline{I}$,
be the first eigenpair to the following cell eigenvalue problem:
\begin{align}\label{eq:auxproblem}
-\mop{div}_y( A(x_1, y) \nabla_y \psi ) & = \mu(x_1) \psi \quad \text{ in } \square(x_1),\\
\psi &= 0 \hskip 1.5cm \text{ on } \partial \square(x_1),\nonumber
\end{align}
normalized by
\begin{align}\label{eq:auxproblemnormalization}
\int_{\square(x_1)} \psi_1(x_1,y)^2 \,dy & = 1.
\end{align}
By the Krein-Rutman theorem, the first eigenvalue $\mu_1(x_1) > 0$ is simple, and $\psi_1(x_1,y)$ does not change sign in $\square(x_1)$, and may for example be chosen positive.
By the regularity of the coefficients $A$ and the domain $\square(x_1)$ in $x_1$,
and the simplicity of the first eigenvalue $\mu_1(x_1)$, one 
has $\mu_1(x_1) \in C^2(\overline{I})$ (see~\cite{gilbarg2015elliptic,komkov1986design}).

We will use $\psi_1$ to construct a test function in the variational formulation for $\lambda_1^\ve$. Namely, we are going to take a test function as a product of $\psi$, taking care of the transversal oscillations, and a function of $x_1$ only. In the proof of Lemma \ref{lm:lambdaestimate} we will see that to minimize the ratio in the variational principle, the latter function should localize in the vicinity of the minimum point of $\mu_1$. This motivates the following structural assumption. 
\begin{align}\label{eq:H}
\mu_1(x_1) \text{ has a unique global minimum at $x_1 = 0$, and $\mu''_1(0)>0$.}
\end{align}
We will come back to this assumption in Section \ref{sec:computing} and show how it is related to the assumptions on the geometry of the cylinder in Theorem~\ref{tm:salmiak}, \cite{FrSo-2009}.

The following a priori estimate holds for the first eigenvalue.

\begin{lemma}\label{lm:lambdaestimate}
Let $\lambda_1^\ve$ be the first eigenvalue to \eqref{eq:originalproblem}, and the first eigenvalue $\mu_1(x_1)$ to \eqref{eq:auxproblem}
satisfies \eqref{eq:H}.
Then
\begin{align*}
\lambda_1^\ve = \frac{\mu_1(0)}{\ve^2} + O\Big(\frac{1}{\ve}\Big), \quad \ve \to 0.
\end{align*}
\end{lemma}
\begin{proof}
In the Rayleigh quotient, 
\begin{align*}
\lambda_1^\ve = \min_{v \in H^1_0(\Oe) \setminus \{ 0 \}} \frac{ \int_{\Oe} A(x_1,\frac{x}{\ve})\nabla v \cdot \nabla v \,dx }{\int_{\Oe} v^2 \,dx},
\end{align*}
we will use test functions of the form
\begin{align*}
v^\ve(x) & = \psi_1\Big(x_1, \frac{x}{\ve}\Big)v\Big(\frac{x_1}{\sqrt{\ve}}\Big),
\end{align*}
where $\psi_1>0$ is the first eigenfunction to \eqref{eq:auxproblem},
and $v \in C^\infty_0(\rr)\setminus\{0\}$ with support contained in $I$. Clearly $v^\ve \in H^1_0(\Oe)$.
The Taylor theorem and hypothesis \eqref{eq:H} give
\begin{align*}
\mu_1(\sqrt{\ve} x_1) & = \mu_1(0) + \ve \mu_1''(0) x_1^2 + o(\ve x_1^2), \quad \ve \to 0.
\end{align*}
Using the boundedness of $a_{ij}$, the regularity properties of $\psi_1$,
re-scaling and integrating by parts, we obtain
\begingroup
\allowdisplaybreaks
\begin{align*}
\lambda_1^\ve
& \le \frac{ \int_{\Omega_\ve} A\big(x_1, \frac{x}{\ve}\big)
\nabla v^\ve \cdot \nabla v^\ve \,dx }{  \int_{\Omega_\ve} (v^\ve)^2 \,dx  } \\
& = \frac{ \ve^{-1}\int_{\rr^d}  (\psi_1^2 A)\big(\sqrt{\ve}x_1, \frac{x}{\sqrt{\ve}}\big) \nabla v \cdot \nabla v \,dx  }{ \int_{\rr^d} \psi_1^2\big(\sqrt{\ve}x_1, \frac{x}{\sqrt{\ve}}\big)  v^2 \,dx } \\
& \quad + \frac{  \int_{\rr^d} (\psi_1 b + \ve^{-1} \psi_1 c + \ve^{-2}\mu_1 \psi_1^2)\big(\sqrt{\ve}x_1, \frac{x}{\sqrt{\ve}}\big) v^2  \,dx }{ \int_{\rr^d} \psi_1^2\big(\sqrt{\ve}x_1, \frac{x}{\sqrt{\ve}}\big) v^2 \,dx } \\
& \le \frac{\mu_1(0)}{\ve^2} + \frac{C}{\ve},
\end{align*}
\endgroup
for some absolute constant $C$ which is independent of $\ve$,
where
\begin{align}
b(x_1,y)  = &- \mop{div}_x(A(x_1,y)\nabla_x \psi_1(x_1,y)), \label{eq:b} \\
c(x_1,y)  = &-  \mop{div}_y(A( x_1, y ) \nabla_x \psi_1(x_1,y)) \label{eq:c}
\\ 
& - \mop{div}_x(A( x_1, y ) \nabla_y \psi_1(x_1,y)).\nonumber
\end{align}

We proceed with the estimate from below.
Let $u_1^\ve$ be the first eigenfunction to~\eqref{eq:originalproblem},
normalized by $\int_{\Oe} (u_1^\ve)^2 \,dx = 1$, 
and let $v_1^\ve$ be such that
\begin{align*}
u_1^\ve(x) = \psi_1\big(x_1,\frac{x}{\ve}\big)v_1^\ve(x).
\end{align*}
Then by the ellipticity condition for $A$, 
by hypothesis \eqref{eq:H} that $\mu_1(0)$ is the minimal value of $\mu_1(x_1)$, and the uniform boundedness of $b(x_1,y)$ and $c(x_1,y)$, we get
\begin{align*}
\lambda_1^\ve & = 
 \int_{\Oe} (\psi_1^2 A)\big(x_1, \frac{x}{\ve}\big) \nabla v_1^\ve \cdot \nabla v_1^\ve  \, dx \\
 & \quad + \int_{\Oe} ( \psi_1 b + \ve^{-1} \psi_1 c  + \ve^{-2}\mu_1 \psi_1^2)\big(x_1,\frac{x}{\ve}\big) (v_1^\ve)^2  \, dx \\
 & \ge \frac{\mu_1(0)}{\ve^2} - \frac{C}{\ve},
\end{align*}
which gives the desired estimate.
\end{proof}

It follows from the estimate in Lemma~\ref{lm:lambdaestimate} that
the first eigenfunction $u_1^\ve$ concentrates (or localizes)
in the vicinity of the minimum point of $\mu_1(x_1)$ (here $x_1=0$).
Namely, for any $\gamma > 0$, 
\begin{align*}
\int_{\Oe\setminus B_\gamma(0)}(u_1^\ve)^2 \, dx & < \gamma, 
\end{align*}
for all small enough $\ve$,
where $B_\gamma(0)$ is the ball of radius $\gamma$ centered at the origin (c.f.~Lemma 3.3 in~\cite{ChPaPi-13}).

\subsection{Spectral asymptotics}

We introduce the homogenized problem for eigenpairs $(\nu, v)$ on $\mathbf{R}$:
\begin{align}\label{eq:limitproblem}
- \frac{d}{dz_1}\big( a^{\mathrm{eff}} \frac{d v}{dz_1} \big) + (c^\mathrm{eff} + \frac{1}{2}\mu_1''(0) z^2_1 )v
& = \nu v, \quad z_1 \in \rr,
\end{align}
where the constant coefficients $a^\mathrm{eff}$, $c^\mathrm{eff}$ are defined as follows.
Again, $\mu_1, \psi_1$ is the first eigenpair to \eqref{eq:auxproblem} normalized by \eqref{eq:auxproblemnormalization}.
Let
\begin{align*}
c^\mathrm{eff} & = -\int_{\square(0)} \psi_1(0,y) \big(  \mop{div}_y(A( x_1, y ) \nabla_x \psi_1(x_1,y))  \\
& \qquad\qquad\qquad\qquad\,\,                           + \mop{div}_x(A( x_1, y ) \nabla_y \psi_1(x_1,y))  \big) \Big|_{x_1 = 0} \,dy,
\end{align*}
and the effective coefficient $a^\mathrm{eff}$ is such that $a^\mathrm{eff} > 0$, and it is given by
\begin{align}
\label{eq:a^eff}
a^\mathrm{eff} & = \int_{\square(0)} \sum_{j=1}^d \psi_1^2(0,y)a_{1j}(0,y) (\delta_{1j} + \partial_{y_j}N(y)) \, dy,
\end{align}
with $N \in H^1(\square(0),\psi_1^2(0,y))$ the unique solution such that
\begin{align*}
\int_{\square(0)} N^2(y) \psi_1^2(0,y) \,dy & = 0,
\end{align*}
to the following auxiliary cell problem:
\begin{align}\label{eq:cellproblemN1}
-\mop{div}_{y} ( (\psi_1^2 A)(0,y) \nabla_y N ) & = \sum_{j=1}^d \partial_{y_j} (\psi_1^2 a_{j1})(0,y), \quad y \in \square(0).
\end{align}

In the definition of the effective coefficient $a^\eff$,
the weighted Sobolev space $H^1(\square(0),\psi_1^2(0,y))$ is used.
Denote
\begin{align*}
L^2(\square(0),\psi_1^2(0,y))
& = \Big\{ v : \int_{\square(0)} v^2(y) \psi_1^2(0,y) \,dy  < \infty \Big\},
\end{align*}
the weighted Lebesgue space.
Then $H^1(\square(0),\psi_1^2(0,y))$ is defined as
\begin{align*}
& H^1(\square(0),\psi_1^2(0,y)) \\ 
& \quad = \{ v \in L^2(\square(0),\psi_1^2(0,y)) : \nabla v \in L^2(\square(0),\psi_1^2(0,y)) \},
\end{align*}
with inner product
\begin{align*}
(u,v)_{H^1(\square(0),\psi_1^2(0,y))} & = \int_{\square(0)} u \, v\,  \psi_1^2(0,y) \,dy
+ \int_{\square(0)} (\nabla_y u \cdot \nabla_y v) \psi_1^2(0,y) \,dy.
\end{align*}
One notes that $\psi_1^2(0,y) > 0$ for $y \in \square(0)$ and $\psi_1^2(0,y) = 0$
for $y$ on $\partial \square(0)$.
Moreover, 
\begin{align}\label{eq:weightintegrability}
\psi_1^2(0,y) & \in L^1_{\mathrm{loc}}(\square(0)), \qquad
\frac{1}{\psi_1^2(0,y)} \in L^1_{\mathrm{loc}}(\square(0)).
\end{align}
The first property in \eqref{eq:weightintegrability} ensures that $C^\infty(\square(0))$
belongs to\linebreak[4] $H^1(\square(0),\psi_1^2(0,y))$,
and the second property in \eqref{eq:weightintegrability} ensures that\linebreak[4] $H^1(\square(0),\psi_1^2(0,y))$ is a
Hilbert space (c.f. \cite{necas,kufner1987some}).


In the cell problem \eqref{eq:cellproblemN1}, the classical non-homogeneous
Neumann condition is not there because the weight $\psi_1^2(0,y)$ effectively
removes the boundary $\partial \square(0)$.

\begin{lemma}
The spectrum of the harmonic oscillator problem (\ref{eq:limitproblem}) is discrete
\[
\nu_1 < \nu_2 < \ldots < \nu_j \to \infty, \quad j\to \infty.
\]
The corresponding eigenfunctions $v_j \in L^2(\rr)$ can be normalized by 
\begin{align}
\label{eq:normalization-v}
\int_{\rr} v_i v_j \, dx_1 = \delta_{ij}.
\end{align}
\end{lemma}

\begin{theorem}\label{tm:one}
Suppose that~\eqref{eq:H} holds.
Let $\lambda_i^\ve$, $u^\ve_i$ be the $i$th eigenpair to \eqref{eq:originalproblem},
normalized by \eqref{eq:normalization-eigenfunctions}, and 
 $\mu_1$, $\psi_1$ be the first eigenpair to \eqref{eq:auxproblem}, 
$\psi_1$ normalized by \eqref{eq:Psi1normalization}.
Then
\begin{align*}
&(i)&&\lambda_i^\ve  = \frac{\mu_1(0)}{\ve^2} + \frac{\lambda_i^0}{\ve} + o\Big( \frac{1}{\ve} \Big), \quad \ve \to 0, \\
&(ii)&&\frac{1}{\ve^{d-1} \sqrt{\ve}}\int_{\Oe} \Big|u^\ve_i(x) -\psi_1\Big(0, \frac{x}{\ve} \Big)v_i^0\Big( \frac{x_1}{\sqrt{\ve}}\Big)\Big|^2\, dx  \to 0, \quad \ve \to 0.
\end{align*}
Here $(\nu_i^0, v_i^0)$ is the $i$th eigenpair to \eqref{eq:limitproblem}, normalized by \eqref{eq:normalization-v}. 

\end{theorem}

The Dirichlet condition for $u^\ve_1$ on the lateral boundary of the cylinder $\Oe$
is captured in the limit by the radial profile $\psi_1(0,y)$ solving \eqref{eq:auxproblem}, 
while the Dirichlet condition on the ends/bases of the cylinder $\Oe$
is translated into exponential decay of the longitudinal profile $v$ solving \eqref{eq:limitproblem} as $|z_1| \to \infty$.

\section{Proof of Theorem \ref{tm:one}}
\subsection{Rescaled and shifted problem}
Lemma \ref{lm:lambdaestimate} suggests studying the asymptotics of
\begin{align*}
\lambda_1^\ve - \frac{\mu_1(0)}{\ve^{2}} = O(1),
\end{align*}
as $\ve$ tends to zero. Moreover, the first eigenfunction tends to localize in the vicinity of $x_1=0$, which is the minimum point of $\mu_1(x_1)$.
Let us subtract $\mu_1(0)/\ve^2$ from both sides of the equation (\ref{eq:originalproblem}), shifting the spectrum to the left, and make the following change of unknowns:
\begin{align}\label{eq:newvariables}
z & = \frac{x}{\sqrt{\ve}}, &
u^\ve ( \sqrt{\ve}z ) & = \psi_1\Big(  \sqrt{\ve}z_1, \frac{z}{\sqrt{\ve}}  \Big)v^\ve(z)\equiv (\psi_1^\ve(z))^2 v^\ve(z) .
\end{align}
The corresponding problem in the up-scaled domain
\begin{align*}
\widetilde{\Oe} = \frac{1}{\sqrt{\ve}} \Oe,    
\end{align*}
takes the form
\begin{align}
 - \mop{div}_z \big( \widetilde{A^\ve} \nabla_z v^\ve \big)
 +  \big(C^\ve& + \frac{ \mu_1(\sqrt \ve z_1) - \mu_1(0) }{\ve}(\psi_1^\ve)^2 \big) v^\ve \notag\\
& = \nu^\ve (\psi_1^\ve)^2 v^\ve \qquad\, \text{ in } \widetilde{\Oe},\label{eq:rescaled}\\
v^\ve & = 0 \qquad\qquad\qquad \text{ on } \widetilde{\Sigma_\ve}.\label{eq:rescaled3}
\end{align}
Here, the coefficients and the potential are given by
\begin{align}
\label{eq:potential}
\widetilde{A^\ve}(z)  & = \widetilde{A}\big( \sqrt{\ve}z_1, \frac{z}{\sqrt{\ve}} \big) =
\psi_1^2\big( \sqrt{\ve}z_1, \frac{z}{\sqrt{\ve}} \big) A\big(\sqrt{\ve}z_1, \frac{z}{\sqrt{\ve}}\big)
\end{align}
\begin{align}
    \label{eq:bigCeff}
C^\ve(z) & = C\big( \sqrt{\ve}z_1, \frac{z}{\sqrt{\ve}} \big) \\
 & = \Big[  - \psi_1(x_1,y) \mop{div}_y(A( x_1, y ) \nabla_x \psi_1(x_1,y))  \notag\\
 & \qquad   - \psi_1(x_1,y) \mop{div}_x(A( x_1, y ) \nabla_y \psi_1(x_1,y))  \notag\\
 & \qquad   - \ve \psi_1(x_1,y) \mop{div}_x(A(x_1,y)\nabla_x \psi_1(x_1,y))\Big] \big( \sqrt{\ve}z_1, \frac{z}{\sqrt{\ve}} \big),\notag
\end{align}
and
\begin{align}\label{eq:nuvedef}
\nu^\ve & =  \ve \lambda^\ve - \frac{\mu_1(0)}{\ve}. 
\end{align}

The functions $v_i^\ve$ and the number $\nu_i^\ve$ are well-defined by \eqref{eq:newvariables} and \eqref{eq:nuvedef}, in terms of $\lambda_i^\ve$, $u_i^\ve$ and $\psi_1$,
and so $\nu_i^\ve, v_i^\ve$ is an eigenpair to the problem \eqref{eq:rescaled}--\eqref{eq:rescaled3}.
The problem for $v^\ve$ \eqref{eq:rescaled}--\eqref{eq:rescaled3} is well-posed in
the weighted Sobolev space $H^1(\widetilde{\Oe}, (\psi_1^\ve)^2, \widetilde{\Sigma_\ve} )$
defined as follows.
Denote
\begin{align*}
L^2(\widetilde{\Oe},(\psi_1^\ve)^2)
& = \Big\{ v : \int_{\widetilde{\Oe}} v^2(z) (\psi_1^\ve)^2 \,dz  < \infty \Big\},
\end{align*}
the weighted Lebesgue space.
Then $H^1(\widetilde{\Oe}, (\psi_1^\ve)^2, \widetilde{\Sigma_\ve} )$ is defined as
\begin{align*}
& H^1(\widetilde{\Oe}, (\psi_1^\ve)^2, \widetilde{\Sigma_\ve} ) \\ 
& \quad = \{ v \in L^2(\widetilde{\Oe},(\psi_1^\ve)^2) : \nabla v \in L^2(\widetilde{\Oe},(\psi_1^\ve)^2)^d,\, v = 0 \text{ on } \widetilde{\Sigma_\ve} \},
\end{align*}
with inner product
\begin{align*}
(u,v)_{H^1(\widetilde{\Oe}, (\psi_1^\ve)^2, \widetilde{\Sigma_\ve})} & = \int_{\widetilde{\Oe}} u\, v\, (\psi_1^\ve)^2 \,dz
+ \int_{\widetilde{\Oe}} \nabla_z u \cdot \nabla_z v \, (\psi_1^\ve)^2 \,dz.
\end{align*}
One notes that $(\psi_1^\ve)^2 > 0$ a.e. on the ends $\widetilde{\Sigma_\ve}$,
so the usual trace mapping can be used in the definition of $H^1(\widetilde{\Oe}, (\psi_1^\ve)^2, \widetilde{\Sigma_\ve} )$.
Moreover, as $(\psi_1^\ve)^2 > 0$ for $z \in \widetilde{\Oe}$ and $(\psi_1^\ve)^2 = 0$
on the lateral boundary $\partial \widetilde{\Oe} \setminus \widetilde{\Sigma_\ve}$.
Furthermore, 
\begin{align}\label{eq:weightintegrabilityve}
(\psi_1^\ve)^2 & \in L^1_{\mathrm{loc}}(\widetilde{\Oe}), \qquad
\frac{1}{(\psi_1^\ve)^2} \in L^1_{\mathrm{loc}}(\widetilde{\Oe}).
\end{align}
The first property in \eqref{eq:weightintegrabilityve} ensures that $C^\infty(\widetilde{\Oe})$
belongs to\linebreak[4] $H^1(\widetilde{\Oe}, (\psi_1^\ve)^2, \widetilde{\Sigma_\ve})$,
and the second property in \eqref{eq:weightintegrabilityve} ensures that\linebreak[4] $H^1(\widetilde{\Oe}, (\psi_1^\ve)^2, \widetilde{\Sigma_\ve})$ is a Hilbert space (c.f. \cite{necas,kufner1987some}).

Note that $C^\ve$ is uniformly bounded thanks to the regularity properties of $\psi_1$, and consequently adding a potential $C v^\ve$ to (\ref{eq:rescaled}) shifts the spectrum by $C$ and makes the potential positive.

\begin{lemma}
The spectrum of problem (\ref{eq:rescaled}) is discrete, bounded from below, and consists of a countably infinite number of points:
\begin{align*}
C \le \nu_1^\ve < \nu_2^\ve \le \nu_3^\ve \le \ldots \le v_i^\ve \to \infty,    
\end{align*}
counted as many times as their finite multiplicity,
with the corresponding eigenfunctions $v_i^\ve$ forming a Hilbert basis in $L^2(\widetilde{\Oe}, (\psi_1^\ve)^2, \widetilde{\Sigma_\ve})$, normalized by
\begin{align}
\label{eq:normalization-v^eps}
\ve^{-(d-1)/2} \int_{\widetilde{\Oe}} |\square(\sqrt \ve z)|^{-1} v_i^\ve v_j^\ve\, (\psi_1^\ve)^2 \, dz = \delta_{ij}.
\end{align}
\end{lemma}

We might therefore obtain information about the asymptotics of
the eigenvalues and eigenfunctions to \eqref{eq:originalproblem}
by studying the spectral asymptotics of the problem \eqref{eq:rescaled}--\eqref{eq:rescaled3},
or view it as a change of variables in studying \eqref{eq:originalproblem}.
In any point of view the following a priori estimate
for the eigenfuction $v_1^\ve$ to the rescaled problem 
is a starting point of analysis.


\subsection{A priori estimates for the rescaled problem}
Since we are dealing with a dimension reduction problem, it is convenient to work in spaces with measure charging the thin cylinder $\widetilde{\Oe}$. Namely, let us denote for brevity
\begin{align}
\label{eq:measure}
d\mu_\ve(z) & = \ve^{-(d-1)/2}|\square(0)|^{-1}\chi_{\widetilde{\Oe}}(z) dz.
\end{align}
Roughly speaking, we divide by the measure of the cross-section, and the rescaled measure will converge to the one-dimensional Lebegues measure charging the real line.  

\begin{lemma}
\label{lm:conv-measures}
The measure $\mu_\ve$ defined by (\ref{eq:measure}) converges weak$^\ast$ in the space of Radon measures $\mathcal M(\rr^d)$, as $\ve \to 0$,  to the measure $\mu$ defined by
\begin{align*}
d\mu_\ve \rightharpoonup dz_1 \times \delta(z'), \quad z = (z_1, z') \in \mathbf{R}^d, \quad \ve \to 0.
\end{align*}
\end{lemma}

\begin{proof}
The proof is similar to one of Lemma 2.1 in \cite{pettersson2017two} (see also the convergence result in Section 3.2 \cite{pankratova2015spectral}).

Let $\varphi \in C_0(\rr^d)$.
Then
\begin{align*}
 \int_{\rr^d} \varphi(z) \, d\mu_\ve(z)
 =
\int_{\ve^{-1/2}I} \frac{\ve^{-\frac{(d-1)}{2}}}{|\square(0)|}\int_{\sqrt\ve Q(\sqrt \ve z_1,\frac{z_1}{\sqrt\ve})} \varphi(z) dz' dz_1.
\end{align*}
Rescaling $y'=z'/\sqrt\ve$ gives
\begin{align*}
 & \frac{1}{|\Box(0)|}\int_{\rr^d} \varphi(z) \, d\mu_\ve(z) \\
 & \quad =
\frac{1}{|\Box(0)|}\int_{\ve^{-1/2}I} \frac{1}{|\square(0)|}\int_{Q(\sqrt \ve z_1,\frac{z_1}{\sqrt\ve})} \varphi(z_1,\sqrt \ve y') dy' dz_1.
\end{align*}
Let us divide the interval $\ve^{-1/2}I$ into  subintervals (translated scaled periods) $I_j^\ve = \sqrt \ve [0,1)+\sqrt \ve j$, $j\in \mathbf Z$. 

On each interval we use the mean-value theorem choosing a point $\sqrt \ve \xi_j$ and get
\begin{align*}
&\frac{1}{|\Box(0)|}\sum_{j} \int_{I_j^\ve} \int_{Q(\sqrt \ve z_1,\frac{z_1}{\sqrt\ve})} \varphi(z_1,\sqrt \ve y') dy' dz_1\\
& \quad =\frac{1}{|\Box(0)|}\sum_{j} \int_{I_j^\ve} \int_{Q(\sqrt\ve  \xi_j z_1,\frac{z_1}{\sqrt\ve})} \varphi(\xi_j,\sqrt \ve y') dy' dz_1. 
\end{align*}
Since $Q(x_1, y_1)$ is periodic with respect to $y_1$, rescaling $z_1=y_1\sqrt \ve$ yields
\begin{align*}
&\frac{1}{|\Box(0)|}\sum_{j} \sqrt \ve \int_{\mathbf T^1} \int_{Q(\sqrt \ve \xi_j, y_1)} \varphi(\xi_j,\ve y') dy' dy_1\\
& \quad =
 \frac{1}{|\Box(0)|}\sum_{j}  \sqrt \ve\int_{\Box(\sqrt \ve \xi_j)} \varphi(\xi_j,\ve y') dy.
\end{align*}
The last sum is a Riemann sum converging to the following integral
\begin{align*}
& \frac{1}{|\Box(0)|}\sum_{j}  \sqrt \ve\int_{\Box(\sqrt \ve \xi_j)} \varphi(\xi_j,\ve y') dy \\
& \quad \to 
\frac{1}{|\Box(0)|}\ii{\rr}  \ii{\Box(0)} \varphi(z_1,0) \,dy dz_1 \\
& \quad = \ii{I}\varphi(z_1,0)\, dz_1 = \ii{\rr^d} \varphi(z) \, d\mu,
\end{align*}
as $\ve$ tends to zero.
\end{proof}

\begin{lemma}\label{lm:v1veestimate}
Suppose that $\mu_1(x_1)$ has a unique global minimum point at $x_1=0$, i.e. \eqref{eq:H} holds.
Let $v^\ve_1$ be the first eigenfunction to (\ref{eq:rescaled}), normalized by (\ref{eq:normalization-v^eps}).
Then we have the following estimates in spaces with measure:
\begin{align*}
\| \psi_1^\ve \nabla v^\ve_1  \|_{L^2(\rr^d, d\mu_\ve)}
+ \| \psi_1^\ve v^\ve_1  \|_{L^2(\rr^d, d\mu_\ve)}
+ \| \psi_1^\ve z_1 v^\ve_1  \|_{L^2(\rr^d, d\mu_\ve)}
& \le C.
\end{align*}
\end{lemma}
\begin{proof}


The weak form of the equation for $\nu^\ve, v^\ve$ is
\begin{align}\label{eq:weak-rescaled}
& \int_{\widetilde{\Oe}} \widetilde{A^\ve}\nabla v^\ve \cdot \nabla \varphi \,dz 
 + \int_{\widetilde{\Oe}} \Big( C^\ve &+ \frac{\mu_1(\sqrt{\ve}z_1) - \mu_1(0)}{\ve}(\psi_1^\ve)^2 \Big)\, v^\ve\, \varphi \,dz \notag \\
& \quad = \nu^\ve \int_{\widetilde{\Oe}}  (\psi_1^\ve)^2\, v^\ve \, \varphi \,dz,
\end{align}
for all $\varphi \in  H^1(\widetilde{\Oe}, (\psi_1^\ve)^2, \widetilde{\Sigma_\ve} )$.

We turn to the a priori estimates for the first eigenfunction $v_1^\ve$,
under the following normalization:
\begin{align*}
\int_{\mathbf{R}^d}
(\psi_1^\ve)^2
(v_1^\ve)^2 \,d\mu_\ve = 1.
\end{align*}
By Lemma~\ref{lm:lambdaestimate},
\begin{align*}
\nu_1^\ve & = \ve \lambda_1^\ve - \frac{\mu_1(0)}{\ve} 
= O(1), \quad \ve \to 0.
\end{align*}
Taking the first eigenfunction as a test function in (\ref{eq:weak-rescaled}), we get
\begin{align*}
& \int_{\mathbf{R}^d} \psi_1^2\big( \sqrt{\ve}z_1, \frac{z}{\sqrt{\ve}} \big)|\nabla_z v^\ve_1|^2 \, d\mu_{\ve}(z) \\
& \quad \le C
\int_{\mathbf{R}^d} (\psi_1^2A)\big( \sqrt{\ve}z_1, \frac{z}{\sqrt{\ve}} \big) \nabla_z v^\ve_1 \cdot \nabla_z v^\ve_1 \, d\mu_{\ve}(z) \\
& \quad  = C \Big(
\nu_1^\ve \int_{\mathbf{R}^d} \psi_1^2\big( \sqrt{\ve}z_1, \frac{z}{\sqrt{\ve}} \big) (v_1^\ve)^2 \,d\mu_\ve(z) \\
& \qquad  - 
\int_{\mathbf{R}^d} (\psi_1 b + \ve \psi_1 c)\big( \sqrt{\ve}z_1, \frac{z}{\sqrt{\ve}} \big) (v_1^\ve)^2 \,d\mu_\ve(z) \\
& \qquad  - \int_{\mathbf{R}^d} \Big(\frac{\mu_1(\sqrt{\ve}z_1) - \mu_1(0)}{\ve}\psi_1^2 \Big)\big(\sqrt{\ve}z_1, \frac{z}{\sqrt{\ve}}\big) (v_1^\ve)^2 \,d\mu_\ve(z)
\Big) \\
& \quad \le C,
\end{align*}
for some constant $C$ which is independent of $\ve$, where one has 
used the boundedness of $\psi_1$, $b$, and $c$.

By hypothesis \eqref{eq:H}, there exists an absolute constant $C$ that is independent of $\ve$ and such that
\begin{align*}
\frac{\mu_1(\sqrt{\ve}z_1) - \mu_1(0)}{\ve} & \ge C z_1^2,
\quad z_1 \in \frac{1}{\sqrt{\ve}}\overline{I}.
\end{align*}
Indeed, suppose that there exists a sequence $\zeta_j$ such that
\[
\frac{\mu_1(\sqrt{\ve}\zeta_j) - \mu_1(0)}{|\sqrt \ve \zeta_j|^2} \to 0, \quad j\to \infty.
\]
Since $z_1 \in (\sqrt{\ve})^{-1}\overline{I}$, then for each fixed $\ve$, $|\sqrt \ve \zeta_j|^2$ is bounded, which yields
\[
\mu_1(\sqrt{\ve}\zeta_j) \to \mu_1(0), \quad j\to \infty.
\]
Then $\zeta_j \to 0$, by uniqueness of the minimum point $0$.
On the other hand, since $\mu_1''(0)$ is strictly positive by assumption, 
\[
\frac{\mu_1(\sqrt{\ve}\zeta_j) - \mu_1(0)}{|\sqrt{\ve}\zeta_j|^2} = \frac{1}{2} \mu_1''(0) + o(1) > \alpha\mu_1''(0)>0 , \quad \zeta_j \to 0
\]
for some $\alpha >0$, and we arrive to a contradiction. 

Therefore, again by the integral identity for the eigenpair $(\nu_1^\ve, v_1^\ve)$,
\begin{align*}
& \int_{\mathbf{R}^d} \psi_1^2\big(\sqrt{\ve}z_1, \frac{z}{\sqrt{\ve}}\big) z_1^2 (v_1^\ve)^2 \,d\mu_\ve(z) \\
& \quad \le
C_1
\int_{\mathbf{R}^d} \psi_1^2\big(\sqrt{\ve}z_1, \frac{z}{\sqrt{\ve}}\big)
\frac{\mu_1(\sqrt{\ve}z_1) - \mu_1(0)}{\ve}
(v_1^\ve)^2 \,d\mu_\ve(z) \\
& \quad = 
C_2 \Big(
\nu_1^\ve \int_{\mathbf{R}^d} \psi_1^2\big(\sqrt{\ve}z_1, \frac{z}{\sqrt{\ve}}\big) (v_1^\ve)^2 \,d\mu_\ve(z) \\
& \qquad\qquad - \int_{\mathbf{R}^d} 
(\psi_1^2A)\big(\sqrt{\ve}z_1, \frac{z}{\sqrt{\ve}}\big)
\nabla_z v_1^\ve \cdot \nabla_z v_1^\ve 
\,d\mu_\ve(z) \\
&\qquad\qquad 
- \int_{\mathbf{R}^d} 
(\psi_1 b + \ve \psi_1 c)\big(\sqrt{\ve}z_1, \frac{z}{\sqrt{\ve}}\big)
(v_1^\ve)^2 
\,d\mu_\ve(z)
\Big) \\
& \quad \le C,
\end{align*}
for some constant $C$ which is independent of $\ve$, by the 
the boundedness of $A$, $\psi_1$, $b$, and $c$, and
the above estimate for $\psi_1\big( \sqrt{\ve}z_1, \frac{z}{\sqrt{\ve}} \big) \nabla_z v_1^\ve$.
Lemma \ref{lm:v1veestimate} is proved. 
\end{proof}

\subsection{Two-scale convergence}

Let us formulate the definition of two-scale convergence for our particular setting.
For the specific weakly convergent Radon measures in \eqref{eq:measure}, as $\ve \to 0$,
\begin{align*}
d\mu_\ve(z)
=
\frac{\ve^{-(d-1)/2}}{|\square(\sqrt \ve z_1)|}\chi_{\widetilde{\Oe}}(z) dz
\rightharpoonup d\mu(z) = dz_1 \times d\delta(z'),
\end{align*}
a bounded sequence 
$v_\ve$ in $L^2(\rr^d, d\mu_\ve)$, that is
\begin{align}
\limsup_{\ve \to \infty} \int_{\rr^d} v_\ve^2(z) \,d\mu_\ve(z) < \infty, \label{eq:bdd}
\end{align}
is said to be weakly two-scale convergent at rate $\sqrt{\ve}$ to a function $v = v(z,y) \in L^2(\rr^d \times \square(0), d\mu \times dy)$, $v_\ve \overset{2}{\rightharpoonup} v$, if 
\begin{align*}
& \lim \limits_{\ve \to 0}
\int_{\rr^d} v_\ve(z) \varphi(z) \phi\Big(\frac{z}{\sqrt{\ve}}\Big)  \,d\mu_\ve(z) \\
& \quad = \frac{1}{|\square(0)|} \int_{\rr^d}  \int_{\square(0)} v(z,y) \varphi(z) \phi(y) \, dy \,d\mu(z),
\end{align*}
for any $\phi \in C^\infty(\ttt^1 \times \mathbf{R}^{d-1})$ and any $\varphi \in C^\infty_0(\rr^d)$.

A bounded sequence 
$v_\ve$ in $L^2(\rr^d, d\mu_\ve)$ is said to be strongly two-scale
converging to a function $v \in L^2(\rr^d \times \square(0), d\mu \times dy)$ if 
\begin{align*}
\lim \limits_{\ve \to 0}\int_{\rr^d} v_\ve(z) w_\ve(z)  \,d\mu_\ve(z) = \frac{1}{|\square(0)|} \int_{\rr^d} \int_{\square(0)} v(z,y) w(z,y) \, dy \,d\mu(z),
\end{align*}
for any weakly two-scale converging sequence $w_\ve \overset{2}{\rightharpoonup} w$ in $L^2(\rr^d, d\mu_\ve)$.

The following two-scale compactness principle holds.
\begin{lemma}\label{tm:twoscaleprinciple}
Let $v_\ve$ be bounded in $L^2(\rr^d, d\mu_\ve)$ \eqref{eq:bdd}.
Then, along a subsequence, $v_\ve$ converges weakly two-scale in $L^2(\rr^d, d\mu_\ve)$
to some $v = v(z,y) \in L^2(\rr^d \times \square(0), d\mu \times dy)$.
\end{lemma}


\subsection{Passage to the limit}

\begin{lemma}
Let $(\nu_1^\ve, v_1^\ve)$ be the first eigenpair to \eqref{eq:rescaled}.
Then, along some subsequence, $(\nu_1^\ve, \psi_1^\ve v_1^\ve)$ converge in $\mathbf{R}$ and 
weakly two-scale in $L^2(\rr^d, d\mu_\ve)$, respectively, to a pair $(\nu, \psi_1(0,y)v)$,
where $(\nu, v)$ is an eigenpair to the effective problem \eqref{eq:limitproblem}.
\end{lemma}
Note that at this point we do not claim that $(\nu, v)$ is the first eigenpair of the limit problem. 
\begin{proof}
By the two-scale compactness principle,
using the a priori estimates in Lemma \ref{lm:v1veestimate} for $\psi_1^\ve \, v_1^\ve$, 
there exist $v, w \in L^2(\mathbf{R}^d \times \square(0), dz_1 \times d\delta(z') \times dy)$
such that along some subsequence, still denoted by $\ve$, 
the following weak two-scale convergences hold
in $L^2(\mathbf{R}^d, d\mu_\ve)$:
\begin{align}\label{eq:2scalew}
\psi_1\big( \sqrt{\ve}z_1, \frac{z}{\sqrt{\ve}} \big) v_1^\ve & \overset{2}{\rightharpoonup} w(z_1,y), \\
\psi_1\big( \sqrt{\ve}z_1, \frac{z}{\sqrt{\ve}} \big) \nabla_z v_1^\ve & \overset{2}{\rightharpoonup} p(z_1,y),\label{eq:2scalew2}
\end{align}
as $\ve$ tends to zero, and in particular $w, p \in L^2(\rr \times \square(0))$.
Let $\nu$ be such that $\nu^\ve$ converges to $\nu$ in $\rr$,
restricting to a further subsequence if necessary.

By the boundedness of the gradient of $v_1^\ve$, Lemma \ref{lm:v1veestimate} or \eqref{eq:2scalew}--\eqref{eq:2scalew2},
one notes that necessarily $w = \psi_1(0,y)v(z_1)$ for some function
$v = v(z_1) \in L^2(\rr)$.
To see this one uses the two-scale convergence in \eqref{eq:2scalew} and oscillating test functions of the form 
\begin{align*}
\Phi^\ve(z) & = \sqrt{\ve} \varphi(z) \phi\big( \frac{z}{\sqrt{\ve}} \big),
\end{align*}
where $\varphi \in C^\infty_0(\mathbf{R}^d)$ and $\phi \in C^\infty_0(\ttt^1 \times \mathbf{R}^{d-1})$.
On the one hand,
\begin{align}
& \int_{\rr^d} (\psi_1^2)^\ve v_1^\ve \partial_{z_i} \Phi^\ve \,d\mu_\ve \notag\\
& \quad = 
\int_{\rr^d} (\psi_1^2)^\ve v_1^\ve \big( \sqrt{\ve}(\partial_{z_i}\varphi)\phi(y) + \varphi \partial_{y_i}\phi \big)(\sqrt{\ve}z_1, \frac{z}{\sqrt{\ve}})  \,d\mu_\ve \notag\\
& \quad \to \int_{\mathbf{R}} \frac{1}{|\square(0)|}\int_{\square(0)} \psi_1(0,y) w(z_1,y) \varphi \partial_{y_i}\phi \,dy \,dz_1,\label{eq:k}
\end{align}
as $\ve$ tends to zero, where one has used that $\psi_1^\ve \partial_{z_i} \Phi^\ve$
is strongly two-scale convergent to $\psi_1(0,y)\varphi(z)\partial_{y_i}\phi(y)$ in $L^2(\rr^d, d\mu_\ve)$.
On the other hand,
\begin{align}
\int_{\rr^d} (\psi_1^2)^\ve v_1^\ve \partial_{z_i} \Phi^\ve \,d\mu_\ve
& = - \int_{\rr^d} \partial_{z_i}((\psi_1^2)^\ve v_1^\ve ) \Phi^\ve \,d\mu_\ve.\label{eq:kk}
\end{align}
Because
\begin{align*}
& \partial_{z_i} ( \psi_1^2(\sqrt{\ve}z_1, \frac{z}{\sqrt{\ve}}) v_1^\ve ) \\
& \quad = 
(\psi_1^2)^\ve \partial_{z_i} v_1^\ve
+ 2 \sqrt{\ve} ( \partial_{z_i}\psi_1 )^\ve \psi_1^\ve v_1^\ve 
+ \frac{1}{\sqrt{\ve}} (\partial_{y_i}\psi_1 )^\ve \psi_1^\ve v_1^\ve,
\end{align*}
one has, by the boundedness of $\psi_1^\ve \nabla_z v_1^\ve$,
\begin{align}
\int_{\rr^d} \partial_{z_i}((\psi_1^2)^\ve v_1^\ve ) \Phi^\ve \,d\mu_\ve
& \to 
\int_{\rr} \frac{1}{|\square(0)|} \int_{\square(0)} 2 w \varphi(z) \phi(y) \partial_{y_i}\psi_1(0,y)  \,dy \,dz_1,\label{eq:kkk}
\end{align}
as $\ve$ tends to zero.
It follows from \eqref{eq:k}--\eqref{eq:kkk} that
\begin{align*}
\psi_1(0,y)\nabla_y w & = 2 w \nabla \psi_1(0,y),
\end{align*}
almost everywhere in $\rr \times \square(0)$.
One concludes that,
\begin{align*}
w(z_1, y) & = \psi_1(0,y)v(z_1),
\end{align*}
for some $v = v(z_1) \in L^2(\rr)$.

We proceed to the two-scale limit $p$ of $\psi_1^\ve \nabla_z v_1^\ve$.
Let $\Phi = \Phi(z_1,y)$ be such that
\begin{align}\label{eq:divPhi}
\mop{div}_y( \psi_1^2(0,y)\Phi(z_1,y) ) & = 0 \quad \text{ in } \ttt \times \rr^{d-1}.
\end{align}
Then one the one hand by \eqref{eq:2scalew2},
\begin{align*}
\int_{\rr^d} (\psi_1^2)^\ve \nabla_z v_1^\ve \cdot \Phi^\ve \,d\mu_\ve
& \to \int_{\rr} \frac{1}{|\square(0)|} \int_{\square(0)} \psi_1(0,y)p(z_1,y) \cdot \Phi(0,y) \,dy \,dz_1,
\end{align*}
as $\ve$ tends to zero.
On the other hand by the choice of test functions, compactly supported $\Phi$,
\begin{align*}
& \int_{\rr^d} (\psi_1^2)^\ve \nabla_z v_1^\ve \cdot \Phi^\ve \,d\mu_\ve \\
& \quad = \int_{\rr^d} (\psi_1^2)(0,\frac{z}{\sqrt \ve}) \nabla_z v_1^\ve \cdot \Phi^\ve \,d\mu_\ve + o(1) \\
& \quad = 
- \int_{\rr^d} v_1^\ve (\psi_1^2)^\ve (\mop{div}_z \Phi)^\ve \,d\mu_\ve + o(1) \\
& \quad \to - \int_{\rr} \frac{1}{|\square(0)|} \int_{\square(0)} \psi_1^2(0,y)v(z_1) (\mop{div}_z \Phi)(0,y) \,dy \,dz_1 \\
& \quad = \int_{\rr} \frac{1}{|\square(0)|} \int_{\square(0)} \nabla_z v(z_1) \cdot \psi_1^2(0,y)\Phi(0,y) \,dy \,dz_1,
\end{align*}
as $\ve$ tends to zero, for any $dz_1 \times d\delta(z')$ gradient $\nabla_z v$ of $v$.
Therefore,
\begin{align*}
\int_{\rr} \int_{\square(0)} \Big( \frac{p}{\psi_1(0,y)} + \nabla_z v \Big) \cdot \psi_1^2(0,y)\Phi(0,y) \,dy \,dz_1 & = 0,
\end{align*}
for any solution $\Phi$ to \eqref{eq:divPhi}.
By the solenoidal nature of $\psi_1^2(0,y) \Phi(z_1,y)$,
\begin{align*}
\frac{p}{\psi_1(0,y)} + \nabla_z v & = \nabla_y q,
\end{align*}
for some $q \in L^2(\rr, H^1(\square(0)))$.
It follows that 
\begin{align*}
p & = \psi_1(0,y)( \nabla_z v + \nabla_y q ),
\end{align*}
almost everywhere in $\rr \times \square(0)$, 
for some $q \in L^2(\mathbf{R}, H^1(\square(0)))$.

To sum up,
by the two-scale compactness principle for sequences with bounded gradient,
there exists $v \in L^2(\rr)$ and $q \in L^2(\rr, H^1(\square(0)))$ such that two-scale weakly in
$L^2( \rr^d,  d\mu_\ve )$,
\begin{align*}
\psi_1^\ve v_1^\ve & \overset{2}{\rightharpoonup} \psi_1(0,y)v(z_1),\\
\psi_1^\ve \nabla_z v_1^\ve & \overset{2}{\rightharpoonup} \psi_1(0,y) (\nabla^{dz_1 \times d\delta(z')} v(z_1) + \nabla_y q(z_1,y)),
\end{align*}
as $\ve$ tends to zero.
One notes that
\begin{align*}
\nabla^{dz_1 \times d\delta(z')} v
= 
(\partial_{z_1}v, r),
\end{align*}
where $r \in L^2(\rr)$ is some transverse gradient of $v$ with respect to the measure $dz_1 \times d\delta(z')$.

We are now in a position to pass to the limit in the equation for $\nu^\ve, v^\ve$.
The variational form of the equation for $\nu^\ve, v^\ve$, 
in terms of the measure $d\mu_\ve$,
is
\begin{align}\label{eq:varformdmuve}
& \int_{\rr^d} (\psi_1^2 A)\big( \sqrt{\ve}z_1, \frac{z}{\sqrt{\ve}} \big) \nabla_z v^\ve \cdot \nabla_z \varphi \,d\mu_\ve \notag \\
& \quad + \int_{\rr^d} (\psi_1( b + \ve c ) + \frac{\mu_1(x_1) - \mu_1(0)}{\ve} \psi_1^2 )\big( \sqrt{\ve}z_1, \frac{z}{\sqrt{\ve}} \big) v^\ve \varphi \,d\mu_\ve \notag \\
& = \nu^\ve \int_{\rr^d} \psi_1^2\big( \sqrt{\ve}z_1, \frac{z}{\sqrt{\ve}} \big) v^\ve \varphi \,d\mu_\ve, 
\end{align}
for any $\varphi \in H^1(\rr^d)$.

We will pass to the limit as $\ve$ tends to zero in \eqref{eq:varformdmuve}
using the two-scale converge.
Let $\varphi \in C^\infty_0(\rr^d)$.
Then
\begin{align*}
& \int_{\rr^d} (\psi_1^2 A)\big( \sqrt{\ve}z_1, \frac{z}{\sqrt{\ve}} \big) \nabla_z v^\ve \cdot \nabla_z \varphi \,d\mu_\ve \\
& \quad = \int_{\rr^d} \psi_1\big( \sqrt{\ve}z_1, \frac{z}{\sqrt{\ve}} \big) \nabla_z v^\ve \cdot (\psi_1 A)\big( \sqrt{\ve}z_1, \frac{z}{\sqrt{\ve}} \big) \nabla_z \varphi \,d\mu_\ve \\
& \quad \to \int_{\rr} \int_{\square(0)} p(z_1,y) \cdot (\psi_1 A)(0,y) \nabla_z \varphi(z_1,0) \,dy \, dz_1 \\
& \quad = \int_{\rr} \int_{\square(0)} (\psi_1 A)(0,y)p(z_1,y) \,dy \cdot \nabla_z \varphi(z_1,0) \, dz_1,
\end{align*}
as $\ve$ tends to zero, because
$(\psi_1 A)\big( \sqrt{\ve}z_1, \frac{z}{\sqrt{\ve}} \big) \nabla_z \varphi$
converges strongly two-scale in $L^2(\rr^d, d\mu_\ve)$
to
$(\psi_1 A)(0,y)\nabla_z \varphi (z_1,0) \in L^2(\rr \times \square(0))$.
Because $((b + \ve c)\psi_1)\big( \sqrt{\ve}z_1, \frac{z}{\sqrt{\ve}} \big) \varphi$ 
converges strongly two-scale in $L^2(\rr^d, d\mu_\ve)$ to $b(0,y)\psi_1(0,y)\varphi(z_1,0) \in L^2(\rr \times \square(0))$,
\begin{align*}
& \int_{\rr^d} ( \psi_1( b + \ve c) )\big( \sqrt{\ve}z_1, \frac{z}{\sqrt{\ve}} \big) v^\ve \varphi  \,d\mu_\ve  \\
& \quad \to \int_{\rr} \int_{\square(0)} b(0,y) w(z_1,y)\,dy \, \varphi(z_1,0) \,dz_1,
\end{align*}
as $\ve$ tends to zero.
Because 
$\nu^\ve \psi_1\big( \sqrt{\ve}z_1, \frac{z}{\sqrt{\ve}} \big) \varphi$ 
converges strongly two-scale in $L^2(\rr^d, d\mu_\ve)$ to~$\nu \psi_1(0,y)\varphi(z_1,0) \in L^2(\rr \times \square(0))$,
\begin{align*}
& \nu^\ve \int_{\rr^d} \psi_1^2\big( \sqrt{\ve}z_1, \frac{z}{\sqrt{\ve}} \big) v^\ve \varphi \,d\mu_\ve \\
& \quad \to \nu \int_\rr \int_{\square(0)} \psi_1(0,y)w(z_1,y) \,dy \, \varphi(z_1,0) \,dz_1,
\end{align*}
as $\ve$ tends to zero.
By the Taylor theorem, and hypothesis~\eqref{eq:H},
\begin{align*}
\frac{\mu_1(\sqrt{\ve}z_1) - \mu_1(0)}{\ve} = \frac{1}{2}\mu_1''(0) z_1^2 + o(z_1^2),
\end{align*}
as $\ve$ tends to zero, so by the compact support of $\varphi$,
\begin{align*}
& \int_{\rr^d} \frac{\mu_1(\sqrt{\ve}z_1) - \mu_1(0)}{\ve}\psi_1^2\big( \sqrt{\ve}z_1, \frac{z}{\sqrt{\ve}} \big)  v^\ve \varphi \,d\mu_\ve \\
& \quad = \int_{\rr^d} \frac{1}{2} \mu_1''(0) z_1^2 \psi_1^2\big( \sqrt{\ve}z_1, \frac{z}{\sqrt{\ve}} \big)  v^\ve \varphi  \,d\mu_\ve + o(1) \\
& \quad \to \int_{\rr} \frac{1}{2} \mu_1''(0) z_1^2 \int_{\square(0)}\psi_1(0,y) w(z_1,y) \,dy \, \varphi(z_1,0) \,dz_1,
\end{align*}
as $\ve$ tends to zero,
because $\frac{1}{2}\mu_1''(0) z_1^2 \psi_1\big( \sqrt{\ve}z_1, \frac{z}{\sqrt{\ve}} \big) \varphi$
converges strongly two-scale in $L^2(\rr^d, d\mu_\ve)$ to $\frac{1}{2}\mu_1''(0) z_1^2 \psi_1(0,y) \varphi
\in L^2(\rr \times \square(0))$.
In conclusion,
\begin{align}\label{eq:limit1}
& \int_{\rr} \int_{\square(0)} (\psi_1 A)(0,y)p(z_1,y) \,dy \cdot \nabla_z \varphi(z_1,0) \, dz_1  \notag\\
& \quad +
\int_{\rr} \int_{\square(0)} b(0,y) w(z_1,y)\,dy \, \varphi(z_1,0) \,dz_1 \notag\\
& \quad +
\int_{\rr} \frac{1}{2} \mu_1''(0) z_1^2 \int_{\square(0)}\psi_1(0,y) w(z_1,y) \,dy \, \varphi(z_1,0) \,dz_1 \notag\\
& =
\nu \int_\rr \int_{\square(0)} \psi_1(0,y)w(z_1,y) \,dy \, \varphi(z_1,0) \,dz_1,
\end{align}
for any $\varphi \in C^\infty_0(\rr^d)$.

Using 
\begin{align*}
w(z_1,y) & = \psi_1(0,y)v(z_1), \\
p(z_1, y) & = \psi_1(0,y)\big((\partial_{z_1}v(z_1), r) + \nabla_y q(z_1,y) \big),
\end{align*}
and the normalization of $\psi_1$,
transforms \eqref{eq:limit1} into
\begin{align}\label{eq:aux77}
& \int_{\rr} \int_{\square(0)} (\psi_1^2 A)(0,y)\big((\partial_{z_1}v(z_1), r) + \nabla_y q(z_1,y) \big) \,dy \cdot \nabla_z \varphi(z_1,0) \, dz_1  \notag\\
& \quad +
\int_{\rr} \int_{\square(0)} (\psi_1b)(0,y) \,dy \,  v(z_1) \varphi(z_1,0) \,dz_1 \notag\\
& \quad +
\int_{\rr} \frac{1}{2} \mu_1''(0) z_1^2 v(z_1) \varphi(z_1,0) \,dz_1 \notag\\
& =
\nu \int_\rr v(z_1) \varphi(z_1,0) \,dz_1,
\end{align}
for any $\varphi \in C^\infty_0(\rr^d)$ (should be $C^\infty(\rr^d)$ if possible).

To compute $q$ one uses oscillating test functions of the form
\begin{align*}
\Phi^\ve(z) & = \sqrt{\ve} \phi(z) \varphi\big( \frac{z}{\sqrt{\ve}} \big),
\end{align*}
with $\phi \in C^\infty_0(\rr^d)$ and $\varphi \in C^\infty( \overline{\square(0)} )$.
One has
\begin{align*}
\nabla_z \Phi^\ve(z) & = \phi(z) \nabla_y\varphi\big( \frac{z}{\sqrt{\ve}} \big)
+ \sqrt{\ve} \varphi\big( \frac{z}{\sqrt{\ve}} \big)\nabla_z \phi(z) .
\end{align*}
By passing to the limit as $\ve$ tends to zero in \eqref{eq:varformdmuve} with test functions $\Phi^\ve$, using their fixed compact support, one obtains from the weak two-scale convergence of 
$\psi_1\big( \sqrt{\ve}z_1, \frac{z}{\sqrt{\ve}} \big)v_1^\ve$,
\begin{align}\label{eq:aux4}
\int_{\rr} \int_{\square(0)} (\psi_1^2A)(0,y) ((\partial_{z_1}v, r) + \nabla_y q) \cdot \nabla_y \varphi  \,dy \, \phi(z_1,0)  \,dz_1 & = 0,
\end{align}
for any $\phi \in C^\infty_0(\rr^d)$ and any $\varphi \in C^\infty( \overline{\square(0)} )$.
If
\begin{align}\label{eq:qfactorization}
q(z_1,y) = N(y) \cdot (\partial_{z_1}v(z_1),r),
\end{align}
equation \eqref{eq:aux4} requires for $N(y)$ to satisfy
\begin{align*}
& \int_{\rr} \int_{\square(0)} \sum_{r,j=1}^d (\psi_1^2a_{rj})(0,y) 
\partial_{y_j} N_1 \partial_{y_r} \varphi 
\,dy \, \partial_{z_1}v \phi(z_1,0) \,dz_1 \\
& \quad
+ \int_{\rr} \int_{\square(0)} \sum_{k=2}^d \sum_{r,j=1}^d (\psi_1^2a_{rj})(0,y) 
\partial_{y_j} N_k \partial_{y_r} \varphi 
\,dy \, r_k \phi(z_1,0) \,dz_1 \\
& =
- \int_{\rr} \int_{\square(0)} \sum_{r=1}^d (\psi_1^2a_{r1})(0,y)\partial_{y_r}
\varphi \,dy \, \partial_{z_1}v \, \phi(z_1,0) \, dz_1 \\
& \quad 
- \int_{\rr} \int_{\square(0)} \sum_{r=1}^d \sum_{j=2}^d (\psi_1^2a_{rj})(0,y)\partial_{y_r}
\varphi \,dy \, r_j \, \phi(z_1,0) \, dz_1,
\end{align*}
for any $\phi \in C^\infty_0(\rr^d)$ and any $\varphi \in C^\infty( \overline{\square(0)} )$.
Let $N_k \in H^1(\square(0),\psi_1^2(0,y))$ be such that
\begin{align*}
\int_{\square(0)} N_k(y) \psi_1^2(0,y) \,dy = 0,
\end{align*}
and 
\begin{align*}
-\mop{div}_y(  (\psi_1^2A)(0,y) \nabla_y N_k  )
& = \sum_{j = 1}^d \partial_{y_j} (\psi_1^2 a_{kj})(0,y), \quad y \in \square(0).
\end{align*}
That is, 
\begin{align}\label{eq:weakNk}
\int_{\square(0)} (\psi_1^2A)(0,y)\nabla_y N_k \cdot \nabla_y \varphi \,dy
& = - \int_{\square(0)} \sum_{j=1}^d (\psi_1^2 a_{kj})(0,y) \frac{\partial \varphi}{\partial y_j} \,dy,
\end{align}
for any $\varphi \in H^1(\square(0),\psi_1^2(0,y))$.
Then $N_k$ are well defined because the bilinear form on the left hand side in \eqref{eq:weakNk} is coercive on\linebreak[4] $H^1(\square(0),\psi_1^2(0,y))/\rr$, by the ellipticity 
condition on $A$ and positivity of $\psi_1^2(0,y)$,
and the compatibility condition is satisfied:
\begin{align*}
\int_{\square(0)} \mop{div}( (\psi_1^2 a_{k})(0,y) ) \,dy
& = \int_{\square(0)} (\psi_1^2 a_{k})(0,y) \cdot \nu \,dy = 0,
\end{align*}
where $a_k$ is the $k$th row of $A$, 
using that $\psi_1(0,y)$ is zero on $\partial \square(0)$.
Therefore, the vector $N$ is such that \eqref{eq:qfactorization} holds,
and in particular,
\begin{align}\label{eq:qlimN}
\psi_1\big( \sqrt{\ve}z_1, \frac{z}{\sqrt{\ve}}\big) \nabla_z v_1^\ve 
& \overset{2}{\rightharpoonup}
\psi_1(0,y)( (\partial_{z_1}v, r) + \nabla_y N \cdot  (\partial_{z_1}v, r) ),
\end{align}
weakly in $L^2(\rr^d, d\mu_\ve)$ as $\ve$ tends to zero.

Let the effective $d \times d$ matrix $A^\eff$ with entries $a^\eff_{ij}$ be defined by 
\begin{align}\label{eq:aeffdef}
a^\eff_{ij} & = \int_{\square(0)} \sum_{k=1}^d(\psi_1^2a_{ik})(0,y)( \delta_{kj} + \partial_{y_k}N_j(y) ) \,dy.
\end{align}
We compute the effective flux $A^\eff (\partial_{z_1}v, r)$.
Let $\varphi(z) = z \cdot \phi(z_1)$, where $\phi(z_1)$ is a vector with components $\phi_j \in C^\infty_0(\rr)$ and $\phi_1(z_1) = 0$.
Then 
\begin{align*}
\varphi & \to 0, \\
\nabla_z \varphi & \to \phi(z_1),
\end{align*}
strongly in $L^2(\rr^d, d\mu_\ve)$, as $\ve$ tends to zero.
By passing to the limit as $\ve$ tends to zero in the variational form \eqref{eq:varformdmuve}
of the equation for $v_1^\ve$ (or equivalently setting $\varphi = z \cdot \phi$ in \eqref{eq:aux77}),
and using the definition of $A^\eff$ and the characterization of the limit of $\nabla_z v_1^\ve$ (\eqref{eq:qfactorization},\eqref{eq:qlimN}), one obtains
\begin{align*}
\int_\rr A^\eff ( \partial_{z_1}v, r ) \cdot \phi(z_1) \,dz_1 & = 0,
\end{align*}
for any $\phi(z_1) \in C^\infty_0(\rr)$.
It follows that the transverse component of the effective flux is zero:
\begin{align}\label{eq:fluxxxx}
A^\eff (\partial_{z_1}v, r)
= \Big( \sum_j a^\eff_{1j} \partial_{z_1}v, 0 \Big)
= A^\eff(\partial_{z_1}v, 0),
\end{align}
recalling that $\phi_1 = 0$.
More precisely,
setting $\varphi = y_i$, $i = 2, \ldots, d$, respectively,
as test functions in the variational form \eqref{eq:weakNk} of the equations for $N_j$ gives
\begin{align*}
\int_{\square(0)} \sum_{k=1}^d (\psi_1^2 a_{ik})(0,y)( \delta_{kj} + \partial_{y_k}N_j(y) ) \,dy = 0,
\end{align*}
for $j = 1, \ldots, d$.
In view of the definition of $A^\eff$ \eqref{eq:aeffdef}, this means that
\begin{align}\label{eq:aeffijiszero}
a^\eff_{ij} & = 0,
\end{align}
for all $(i,j) \neq (1,1)$.
That is, all transverse components of the effective matrix are zero.
It follows that the effective flux in \eqref{eq:fluxxxx} reduces to 
\begin{align*}
A^\eff (\partial_{z_1}v, r)
= ( a^\eff_{11} \partial_{z_1}v, 0).
\end{align*}
One may verify that $a^\eff_{11} > 0$ as follows.
Use $N_i$ as a test function in the variational form \eqref{eq:weakNk} of the equation for $N_j$ to obtain,
\begin{align}\label{eq:aeffwriting}
a^\eff_{ij} & = \int_{\square(0)} (\psi_1^2 A)(0,y)\nabla_y ( y_i + N_i(y) )
\cdot \nabla_y ( y_j + N_j(y) )\,dy.
\end{align}
It follows from \eqref{eq:aeffijiszero} and \eqref{eq:aeffwriting}
that $A^\eff$ is symmetric and positive semidefinite 
by the same properties of $A$.
In particular, by \eqref{eq:aeffwriting},
\begin{align*}
a^\eff_{11}
& = \int_{\square(0)} (\psi_1^2 A)(0,y) \nabla_y ( y_1 + N_1(y)) \cdot \nabla_y ( y_1 + N_1(y))   \,dy \\
& \ge \int_{\square(0)} \psi_1^2(0,y) | \nabla_y ( y_1 + N_1(y) )  |^2 \,dy.
\end{align*}
Suppose that $a_{11}^\eff = 0$.
Then by the last inequality, $\psi_1^2(0,y)\nabla_y (y_1 + N_1(y)) = 0$ a.e. in $\square(0)$,
which by the connectedness of $\square(0)$ and the positivity of $\psi_1^2$
implies that $y_1 + N_1(y)$ is constant,
which contradicts the periodicity of $N_1$ in $y_1$.
Therefore,
\begin{align*}
a_{11}^\eff > 0.
\end{align*}

We show that $\nu, v$ is an eigenpair to the effective equation \eqref{eq:limitproblem}.
By passing to the limit in the variational form \eqref{eq:varformdmuve} of the equation for $v_1^\ve$
(or equivalently reading off from \eqref{eq:aux77}),
using a test function $\varphi \in C^\infty_0(\rr^d)$,
one obtains
\begin{align*}
& \int_{\rr} a^\eff \partial_{z_1}v(z_1) \partial_{z_1} \varphi(z_1,0) \, dz_1  
+ \int_{\rr} \big( c^\eff + \frac{1}{2} \mu_1''(0) z_1^2 \big) v(z_1) \varphi(z_1,0) \,dz_1 \\
& \quad =
\nu \int_\rr v(z_1) \varphi(z_1,0) \,dz_1,
\end{align*}
for any $\varphi \in C^\infty_0(\rr^d)$.
This shows that $\nu, v$ is an eigenpair to the problem~\eqref{eq:limitproblem},
by the density of the traces of $C^\infty_0(\rr^d)$ in $C^\infty_0(\rr)$.
\end{proof}

Now we show that $\nu, v$ is necessarily the first eigenpair to the effective equation
and conclude that full $\ve$ sequence $\nu_1^\ve, v_1^\ve$ converges.

\begin{lemma}
The whole sequence $(\nu^\ve_1, \psi_1^\ve v_1^\ve)$ converges
to $(\nu_1, \psi_1(0,y)v_1)$, in $\mathbf{R}$ and weakly two-scale in $L^2(\rr^d, d\mu_\ve)$, respectively,
where $\nu_1, v_1$ is the first eigenpair of the limit problem \eqref{eq:limitproblem}.
\end{lemma}
\begin{proof}
We show that $\nu, v$ is the first eigenvalue to the limit problem \eqref{eq:limitproblem},
and the whole sequences $\nu^\ve_1$, $v_1^\ve$ converge.

Let $\phi^\ve$ be a smooth cutoff for the interval $I$ such that
\begin{align*}
& \phi^\ve \in C^\infty(\overline{I}),\\
& \phi^\ve = 0 \text{ on } \partial I,\\
& 0 \le \phi^\ve \le 1 \text{ in } I, \\
& \phi^\ve(x_1) = 1 \text{ for } \mop{dist}(x_1,\partial I) > \ve, \\
& \ve |\partial_{x_1}\phi^\ve| \le 1.
\end{align*}
Say,
\begin{align*}
\phi^\ve(x_1) & = 
\begin{cases}
\mop{dist}(\frac{x_1}{\ve}, \partial I), & 0 \le \mop{dist}(x_1,\partial I) \le \ve,\\
1, & \text{otherwise.}
\end{cases}
\end{align*}
Then the function $\phi^\ve(\sqrt{\ve}z_1)$ is smooth and cuts off
the growing interval $\frac{1}{\sqrt{\ve}}I$ in a $\sqrt{\ve}$ neighborhood of its boundary,
with gradient satisfying the estimate
\begin{align*}
\big|\partial_{z_1}\big(\phi^\ve( \sqrt{\ve}z_1 )\big)\big| & \le \frac{1}{\sqrt{\ve}}.
\end{align*}

Let $\nu_1, v_1$ be the first eigenpair to the limit problem
\eqref{eq:limitproblem}.
Now we will consider the following test function in the 
variational principle for the eigenvalue $\nu^\ve$:
\begin{align*}
w_\ve(z) = \phi^\ve(\sqrt{\ve}z_1)
\Big( v_1(z_1) + \sqrt{\ve} N_1\big( \frac{z}{\sqrt{\ve}} \big)\partial_{z_1}v_1(z_1) \Big).
\end{align*}
By the definitions of $\phi^\ve$, 
$v_1$, $N_1$,
one has for any $\ve > 0$, $w^\ve \in H^1(\rr^d) \setminus \{ 0 \}$,
with vanishing trace on the end planes
$\big\{ z = (z_1,z') \in \rr^d : z_1 \in \partial \frac{1}{\sqrt{\ve}}I \big\}$.

Then for $\nu_1^\ve = \ve \lambda_1^\ve - \mu_1(0)/\ve$ one has 
from the variational principle, using the test function $w^\ve$,
\begin{align*}
\nu^\ve_1 & \le \frac{\int_{\rr^d} (\psi_1^2A)\big(\sqrt{\ve}z_1, \frac{z}{\sqrt{\ve}}\big) \nabla_z w_\ve(z) \cdot \nabla_z w_\ve(z) \,d\mu_\ve(z)}{\int_{\rr^d} \psi_1^2\big(\sqrt{\ve}z_1, \frac{z}{\sqrt{\ve}}\big) w_\ve^2(z) \,d\mu_\ve(z)} \\
& \quad + \frac{\int_{\rr^d} (\ve \psi_1 b + \psi_1 c + \ve^{-1}(\mu_1 - \mu_1(0))\psi_1^2)\big(\sqrt{\ve}z_1, \frac{z}{\sqrt{\ve}}\big) w_\ve^2(z) \,d\mu_\ve(z)}{\int_{\rr^d} \psi_1^2\big(\sqrt{\ve}z_1, \frac{z}{\sqrt{\ve}}\big) w_\ve^2(z) \,d\mu_\ve(z)}.
\end{align*}
One has with $z = (z_1,z')$,
\begin{align*}
\partial_{z_1} w_\ve(z) & =  
 ( \partial_{z_1}( \phi^\ve(\ve z_1) ) )( v_1(z_1) + \sqrt{\ve}N_1\big( \sqrt{\ve}z_1, \frac{z}{\sqrt{\ve}} \big) \partial_{z_1}v_1(z_1) ) \\
&\qquad  + \phi^\ve(\sqrt{\ve}z_1)\big( 
\big(1 + \partial_{y_1}N\big(\sqrt{\ve} z_1, \frac{z}{\sqrt{\ve}} \big)\big)\partial_{z_1}v_1(z_1)
\big),\\
\nabla_{z'} w_\ve(z) & = \partial_{z_1}v_1(z_1) \nabla_{y'}N_1\big( \frac{z}{\sqrt{\ve}} \big).
\end{align*}
The following estimates follows:
\begin{align*}
\int_{\rr^d} \psi_1^2\big(\sqrt{\ve}z_1, \frac{z}{\sqrt{\ve}}\big) w_\ve^2(z) \,d\mu_\ve(z) 
& = \int_{\rr} v_1^2(z_1) \,dz_1 + o(1),
\end{align*}
and
\begin{align*}
& \frac{\int_{\rr^d} (\psi_1^2A)\big(\sqrt{\ve}z_1, \frac{z}{\sqrt{\ve}}\big) \nabla_z w_\ve(z) \cdot \nabla_z w_\ve(z) \,d\mu_\ve(z)}{\int_{\rr^d} \psi_1^2\big(\sqrt{\ve}z_1, \frac{z}{\sqrt{\ve}}\big) w_\ve^2(z) \,d\mu_\ve(z)} \\
& \quad \le 
\int_{\rr} a^\eff (\partial_{z_1}v_1(z_1))^2 \,dz_1 + o(1),
\end{align*}
and
\begin{align*}
& \frac{\int_{\rr^d} (\ve \psi_1 b + \psi_1 c + \ve^{-1}(\mu_1 - \mu_1(0))\psi_1^2)\big(\sqrt{\ve}z_1, \frac{z}{\sqrt{\ve}}\big) w_\ve^2(z) \,d\mu_\ve(z)}{\int_{\rr^d} \psi_1^2\big(\sqrt{\ve}z_1, \frac{z}{\sqrt{\ve}}\big) w_\ve^2(z) \,d\mu_\ve(z)} \\
& \quad = 
\int_{\rr} \big(c^\eff + \frac{1}{2}z_1^2 \mu_1''(0)\big)v_1^2(z_1) \,dz_1 + o(1),
\end{align*}
as $\ve$ tends to zero.
By the variational principle (or directly from the variational form of the equation with $v_1$ as test function) for the first eigenvalue $\nu_1$ for the limit problem \eqref{eq:limitproblem},
\begin{align*}
\nu_1 & = \frac{\int_{\rr} a^\eff (\partial_{z_1}v_1(z_1))^2 \,dz_1 + \int_{\rr} \big(c^\eff + \frac{1}{2}z_1^2 \mu_1''(0)\big)v_1^2(z_1) \,dz_1}{\int_{\rr} v_1^2(z_1) \,dz_1}.
\end{align*}
One gets the estimate, along a subsequence,
\begin{align*}
\nu_1^\ve & \le \nu_1 + o(1),  
\end{align*}
as $\ve$ tends to zero.
One concludes that for the whole $\ve$ sequence,
\begin{align*}
\lim_{\ve \to 0} \nu_1^\ve = \nu = \nu_1.
\end{align*}
In terms of $\lambda_1^\ve$ this estimate reads,
\begin{align*}
\lambda_1^\ve = \frac{\mu_1(0)}{\ve^2} + \frac{\nu_1}{\ve} + o\big( \frac{1}{\ve} \big),
\end{align*}
as $\ve$ tends to zero.

By the simplicity of the first eigenvalue $\nu_1$ to the limit problem~\eqref{eq:limitproblem},
the whole sequence $v_1^\ve$ converges to $v(z_1) = v_1(z_1)$, the first eigenfunction
to the limit problem~\eqref{eq:limitproblem}.
\end{proof}

Using the fact that the limit of $\nu_1^\ve, v_1^\ve$ is the first eigenpair to the 
effective equation \eqref{eq:limitproblem}, we derive the following 
a priori estimate for the second eigenvalue $\nu_2^\ve$ to the problem
\eqref{eq:rescaled}.

\begin{lemma}
Let $\nu_2^\ve$ be the second eigenvalue to the problem \eqref{eq:rescaled}.
Then
\begin{align*}
\nu_2^\ve & \le \nu_2 + o(1),
\end{align*}
as $\ve$ tends to zero.
\end{lemma}
\begin{proof}
One notes that $v_1^\ve$ is almost orthogonal to $v_2(z_1)$ as $\ve$ tends to zero because
$v_1^\ve$ converges to $v_1(z_1)$ which is orthogonal to $v_2(z_1)$.
Let
\begin{align*}
w_2^\ve & = \phi^\ve(\sqrt{\ve}z_1)
\Big( v_2(z_1) + \sqrt{\ve} N_1\big( \frac{z}{\sqrt{\ve}} \big)\partial_{z_1}v_2(z_1) \Big).
\end{align*}
Then $v_1^\ve$ is almost orthogonal to $w_2^\ve$.
One may therefore use
\begin{align*}
w_2^\ve(z) - \frac{(w_2^\ve(z),v_1^\ve)}{(v_1^\ve,v_1^\ve)}v_1^\ve
\end{align*}
as a test function in the variational principle for $\nu_2^\ve$.
\end{proof}

%


\begin{lemma}
Suppose that for all $i \le k$,
\begin{align*}
\lim_{\ve \to 0} \psi_1\big( \sqrt{\ve}z_1, \frac{z}{\sqrt{\ve}} \big) v_i^\ve = v_i \quad \text{ strongly in } L^2(\rr^d, d\mu_\ve).
\end{align*}
Let $m > k$.
Then $v_m^\ve$ is asymptotically orthogonal to $v_i$ for all $i \le k$:
\begin{align*}
\lim_{\ve \to 0}\int_{\rr^d} \psi_1^2\big( \sqrt{\ve}z_1, \frac{z}{\sqrt{\ve}} \big)
v^\ve_{m}(z) v_i(z_1) \,d\mu_\ve(z)
& = 0,
\end{align*}
and $v_m$ is asymptotically orthogonal to $v_i^\ve$ for all $i \le k$:
\begin{align*}
\lim_{\ve \to 0}\int_{\rr^d} \psi_1^2\big( \sqrt{\ve}z_1, \frac{z}{\sqrt{\ve}} \big)
v_{m}(z_1) v^\ve_i(z) \,d\mu_\ve(z)
& = 0.
\end{align*}
\end{lemma}
\begin{proof}
Let $m > k$.
Then by the orthonormalization \eqref{eq:normalization-v^eps} of the eigenfunctions $v_j^\ve$,
\begin{align*}
& \int_{\rr^d} \psi_1^2\big( \sqrt{\ve}z_1, \frac{z}{\sqrt{\ve}} \big)
v^\ve_{m}(z) v_i(z_1) \,d\mu_\ve(z) \\
& \quad = 
\sum_{i=1}^k \int_{\rr^d} \psi_1^2\big( \sqrt{\ve}z_1, \frac{z}{\sqrt{\ve}} \big)
v^\ve_{m}(z) v^\ve_i(z) \,d\mu_\ve(z) \\
& \qquad +
\sum_{i=1}^k \int_{\rr^d} \psi_1^2\big( \sqrt{\ve}z_1, \frac{z}{\sqrt{\ve}} \big)
v^\ve_{m}(z) (v_i(z_1) - v_i^\ve(z))  \,d\mu_\ve(z) \\
& \quad = 
\sum_{i=1}^k \int_{\rr^d} \psi_1^2\big( \sqrt{\ve}z_1, \frac{z}{\sqrt{\ve}} \big)
v^\ve_{m}(z) (v_i(z_1) - v_i^\ve(z))  \,d\mu_\ve(z).
\end{align*}
By the normalization of $v_m^\ve$, $\psi_1\big(\sqrt{\ve}z_1,\frac{z}{\sqrt{\ve}}\big)v_m^\ve$ is bounded in $L^2(\rr^d, d\mu_\ve)$.
The first asymptotic orthogonality follows from the convergence of $\psi_1\big(\sqrt{\ve}z_1,\frac{z}{\sqrt{\ve}}\big) v_i^\ve $ to $v_i$ in $L^2(\rr^d, d\mu_\ve)$ because
\begin{align*}
\int_{\rr^d} \psi_1^2\big( \sqrt{\ve}z_1, \frac{z}{\sqrt{\ve}} \big)
v^\ve_{m}(z) (v_i(z_1) - v_i^\ve(z))  \,d\mu_\ve(z) = o(1),
\end{align*}
as $\ve$ tends to zero, for any $i \le k$.

The second asserted asymptotic orthogonality follows from the strong convergence and the orthogonality of $v_m$
to $v_i$ for $i \neq m$. 
\end{proof}

We approach the convergence of spectrum by considering the second eigenvalue
for illustration.

\begin{lemma}
$\nu_2^\ve \to \nu_2$, as $\ve \to 0$.
\end{lemma}
\begin{proof}
The second eigenfunction $v_2^\ve$ is almost orthogonal to $v_1^\ve$
when $\ve$ small.
Then $\nu_2^\ve, v_2^\ve$ converges along a subsequence (using estimate in previous lemma for a priori estimates for two-scale convergence) to an eigenpair $\nu, v$ such that
$v$ is almost orthogonal to $v_1$.
By a previous lemma we must have $\nu = \nu_2$,
and thus $v = v_2$, using that all eigenvalues of the limit problem are simple.
\end{proof}

\begin{lemma}
Convergence of the spectrum $(\nu_i^\ve, v_i^\ve)$.
\end{lemma}
\begin{proof}
We know that $(\nu_1^\ve, v_1^\ve) \to (\nu_1, v_1)$.
Suppose that $(\nu_i^\ve, v_i^\ve) \to (\nu_i, v_i)$ for all $i \le k$.
Let
\begin{align*}
w_{k+1}^\ve & = \phi^\ve(\sqrt{\ve}z_1)\Big( v_{k+1}(z_1) + \sqrt{\ve} N_1(\frac{z}{\sqrt{\ve}})\partial_{z_1}v_{k+1}(z_1) \Big).
\end{align*}
One verifies that $w_{k+1}^\ve$ is asymptotically orthogonal to $v_i$ for all $i \le k$, using
the previous lemma.
Then one shows that
\begin{align*}
v_{k+1}^\ve(z) - \sum_{i = 1}^k \frac{(w_{k+1}^\ve(z),v_{i}^\ve(z))}{(v_{i}^\ve(z),v_{i}^\ve(z))} v_{i}^\ve(z)
\end{align*}
are nonzero test functions for all $\ve$.
Use these test functions to obtain the estimate
\begin{align*}
\nu_{k+1}^\ve \le \nu_{k+1} + o(1),
\end{align*}
as $\ve$ tends to zero.
Then one can prove that $\nu_{k+1}^\ve, v_{k+1}^\ve$ converges to some eigenpair.
Use the simplicity and the upper estimate for $\nu_{k+1}^\ve$ to conclude that
for the full $\ve$ sequence
\begin{align*}
\lim_{\ve \to 0} \nu_{k+1}^\ve = \nu_{k+1}.
\end{align*}
This shows by induction that 
\begin{align*}
\lim_{\ve \to 0} \nu_{i}^\ve = \nu_i,
\end{align*}
for any $i$,
and that the corresponding eigenfunctions converge.
\end{proof}

Putting the above sequence of lemmas together, one concludes Theorem~\ref{tm:one}.

Remark that in general for $d\mu_\ve \rightharpoonup d\mu$, and a sequence $v^\ve$
weakly converging in $L^2(\rr^d, d\mu_\ve)$ to $v \in L^2(\rr^d, d\mu)$,
it is not always the case that
\begin{align}\label{eq:stron}
\lim_{\ve \to 0} \int_{\rr^d} (v^\ve - v)^2 \,d\mu_\ve & = 0.
\end{align}
To the positive, for instance, \eqref{eq:stron} holds if $v$ is bounded and the measure of $\rr^d$ is finite in the limit.
In our case the measure of $\rr^d$ is not finite, while the solution
is bounded and exponentially decaying, which compensates.

\section{Computing the leading terms}\label{sec:computing}

In this section we connect hypothesis \eqref{eq:H} with the hypothesis in Theorem~\ref{tm:salmiak}.
We also describe a scheme to compute the effective coefficients and the leading terms in 
the expansions of the eigenpairs in~Theorem~\ref{tm:one}.
The procedure goes as follows:
\begin{enumerate}[1.]
\item Locating the global minimum of the principle eigenvalue $\mu_1$.
\item Computing the effective coefficients $a^\mathrm{eff}$, $c^\mathrm{eff}$.
\item Computing the eigenpair $\lambda_1^0$, $v_1^0$.
\end{enumerate}
Because we do not have an effective characterization of the minimum of principal eigenvalue $\mu_1$,
some iterative procedure could be useful, say the Newton method.
For this purpose we compute the derivative $\mu'_1$ with respect to $x_1$.
The Hessian $\mu_1''$ can be obtained by the similar procedure to that of 
Lemma \ref{lm:mu1prime} below, or a finite difference after computing $\mu_1(x_1)$ at points close to the minimum.

\begin{lemma}\label{lm:mu1prime}
Let $\mu_1, \psi_1$ be the principle eigenpair to the problem
\begin{align}\label{eq:Psidefeq}
-\mop{div}_y(A(x_1, y)\nabla_y \psi) & = \mu(x_1) \psi \quad \text{ in } \square(x_1),\\
\psi & = 0 \quad \hspace*{1.05cm} \text{ on } \partial \square(x_1),\notag
\end{align}
normalized by
\begin{align}\label{eq:Psi1normalization}
\int_{\square(x_1)} \psi^2_1 \,dy = 1.
\end{align}
Let $V_n$ be the outward normal velocity of the boundary $\partial \square(x_1)$, and 
let $V$ be a globally defined velocity field for $\square(x_1)$, with respect to $x_1$.
Then
\begin{align*}
\mu_1' & = \int_{\square(x_1)} \frac{\partial A}{\partial x_1}\nabla \psi_1 \cdot \nabla \psi_1  \,dy 
- \int_{\partial \square(x_1)} (A \nabla_y \psi_1 \cdot \nabla_y \psi_1) (V \cdot \nu)  \,d\sigma \\
& \quad +    
 2 \int_{\square(x_1)} ( A \nabla_y \psi_1 \cdot \nabla_y (\nabla_y \psi_1 \cdot V )
- \mu_1 \psi_1 (\nabla_y \psi_1 \cdot V )   ) \,dy.
\end{align*}
\end{lemma}

To prove Lemma~\ref{lm:mu1prime} we will 
consider separately the contributions to the linearizations from 
the dependence on the coefficients $A(x_1,y)$
and the dependence on the change in shape of $\square(x_1)$.
Lemma~\ref{lm:mu1prime} follows directly from Lemma~\ref{lm:constantcell} and \ref{lm:nonconstantcell} below.

The following example relates hypothesis~\eqref{eq:H} to 
the hypothesis used in~\cite{vishik1957regular} for a smooth profile $h$.
We will make use of the following boundary point property for the first eigenfunction
$\psi_1$ to the problem \eqref{eq:auxproblem}.

\begin{lemma}\label{lm:boundarypointlemma}
Let $\mu_1, \psi_1$ be the first eigenpair to the problem \eqref{eq:auxproblem},
with sign chosen such that $\psi_1(x_1,y) > 0$ everywhere in $\square(x_1)$.
Then
\begin{align*}
\nabla_y \psi_1(x_1,y) \cdot \nu & < 0 \quad \text{ a.e. } y \in \partial \square(x_1),
\end{align*}
for any $x_1 \in \overline{I}$.
\end{lemma}
\begin{proof}
One notes that for any $x \in \overline{I}$, $\psi_1(x_1,y)$ is continuous up to the boundary in $y$, and satisfies
$\psi_1(x_1,y_0) = 0$ for any $y_0 \in \partial \square(x_1)$,
and in particular $\psi_1(x_1,y) > \psi_1(x_1,y_0) = 0$ for every $y \in \square(x_1)$.
Moreover, $\psi_1$ is a subsolution to the equation
\begin{align*}
-\mop{div}_y(A(x_1,y)\nabla_y \psi_1) - \mu_1 \psi_1 = 0, \quad y \in \square(0).
\end{align*}
By the Lipschitz continuity of $\partial \square(x_1)$, 
at almost every $y_0 \in \partial \square(x_1)$, the outward unit normal $\nu$ exists,
the outward normal derivative exists $\nabla_y \psi_1 \cdot \nu$, and
there is a ball $B_R(y) \subset \square(x_1)$ with $y_0 \in \partial B_R(y)$.
The assertion then follows from the classical argument using the weak maximum
principle (c.f. Lemma 3.4 in \cite{gilbarg2015elliptic}).
\end{proof}

\begin{example}\label{ex:hcase}
Consider the case of the Dirichlet Laplacian $-\Delta$
in a finite thin strip in $\mathbf{R}^2$,
with profile given by a smooth positive $h(x_1)$, $x_1 \in [-1,1]$:
\begin{align*}
\Omega_\ve & = \{ x : -1 < x_1 < 1, \, 0 < x_2 < \ve h(x_1)  \}.
\end{align*}
A function $F : [-1,1] \times \mathbf{T}^1 \times \mathbf{R} \to \mathbf{R}$
such that
\begin{align*}
\Omega_\ve = \{ x : x_1 \in (-1,1), \, F(x_1,x/\ve) > 0 \}
\end{align*}
is 
\begin{align*}
F(x_1,y) & = \frac{y_2}{h(x_1)}\Big( 1 - \frac{y_2}{h(x_1)} \Big).
\end{align*}
The corresponding up-scaled cell is
\begin{align*}
\square(x_1) & = \{ y : F(x_1,y) > 0 \} = \{ y : y_1 \in \mathbf{T}^1 : \, 0 < y_2 < h(x_1) \}.
\end{align*}
The normal velocity of $\partial \square(x_1)$ is
\begin{align}\label{eq:normalexample}
V \cdot \nu & = - \frac{\partial_{x_1}F}{|\nabla_y F|} = -\frac{h'  y_2}{h},
\end{align}
and a globally defined smooth domain velocity field on $\overline{\square(x_1)}$ is
\begin{align*}
V & = \Big(0, -\frac{h' y_2}{h}\Big).
\end{align*}
By Lemma~\ref{lm:mu1prime}, 
\begin{align*}
\mu_1' & =  
- \int_{\partial \square(x_1)} |\nabla_y \psi_1|^2 (V \cdot \nu)  \,d\sigma \\
& \quad +   2 \int_{\square(x_1)} (( \nabla_y \psi_1 \cdot \nu) \nabla_y (\nabla_y \psi_1 \cdot V )
- \mu_1 \psi_1 (\nabla_y \psi_1 \cdot V )   ) \,dy \\
& = 
- \int_{\partial \square(x_1)} |\nabla_y \psi_1|^2 (V \cdot \nu)  \,d\sigma 
+ 2 \int_{\partial\square(x_1)} ( \nabla_y \psi_1 \cdot \nu) (\nabla_y \psi_1 \cdot V )  \,d\sigma \\
& = \int_{\partial \square(x_1)} |\nabla_y \psi_1|^2 (V \cdot \nu) \,d\sigma,
\end{align*}
where one in the second step has used the divergence theorem and the equation~\eqref{eq:Psidefeq},
and in the third step has used that here
$(\xi \cdot \nu)(\xi \cdot V) = |\xi|^2 (V \cdot \nu)$, $\xi \in \mathbf{T}^1 \times \mathbf{R}$.
One has $\int_{\partial \square(x_1)} |\nabla_y \psi_1|^2 \,d\sigma > 0$,
by the boundary point property, Lemma~\ref{lm:boundarypointlemma}, irregardless of the sign chosen for $\psi_1$.
(In other words, because otherwise the critical set
$\{ y \in \overline{\square(x_1)} : \psi_1 = 0, \, \nabla_y \psi_1 = 0 \}$
would be of positive $(d-1)$-dimensional measure.)
It follows that $\mu_1'(x_1) = 0$ if and only if $h'(x_1) = 0$.
Therefore by \eqref{eq:normalexample},
the hypothesis of unique minimum of $\mu_1(x_1)$ is 
equivalent to the existence of a unique maximum of $h(x_1)$ in this example.
One might remark that in this example, one also has access to both the domain monotonicity of the eigenvalues, as well as the exact eigenpair, none of which is available if $A(x_1,y)$ is not constant in the fast variable $y$.
\end{example}

We conclude this section by computing the linearization given in Lemma~\ref{lm:mu1prime}.
We first compute $\mu_1'$ for a constant cell $\square$,
that is $\square(x_1)$ independent of $x_1$.

\begin{lemma}\label{lm:constantcell}
Suppose that $\square(x_1) = \square$ is independent of $x_1$.
Let $A(x_1,y)$ and $\partial \square$ be sufficiently smooth.
Let $\mu_1, \psi_1$ be the principle eigenpair to
\begin{align*}
-\mop{div}_y (A(x_1,y) \nabla_y \psi ) & = \mu(x_1) \psi \quad \text{ in } \square,
\end{align*}
with the homogeneous Dirichlet condition on $\partial \square$, and normalized by
\begin{align*}
\int_\square \psi^2_1 \,dy = 1.
\end{align*}
Then by the Fr{\'e}chet differentiability of the eigenpair and the bilinear forms associated to the problem, remarking that $\mu_1$ is simple, 
\begin{align*}
\mu'_1 & = \int_\square \frac{\partial A}{\partial x_1}\nabla_y \psi_1 \cdot \nabla_y \psi_1  \,dy.
\end{align*}
\end{lemma}
\begin{proof}
For any test function $\varphi \in H^1_0(\square)$,
\begin{align}\label{eq:odin}
\int_\square A\nabla_y \psi_1\cdot \nabla_y \varphi \,dy & = \mu_1 \int_\square \psi_1 \varphi \,dy.
\end{align}
Differentiate both sides with respect to $x_1$, to obtain
\begin{align}\label{eq:dva}
&  \int_\square \frac{\partial A}{\partial x_1}\nabla_y \psi_1\cdot \nabla_y \varphi \,dy \notag
+ \int_\square A\nabla_y \frac{\partial \psi_1}{\partial x_1} \cdot \nabla_y \varphi \,dy  \\
&\quad = \mu_1' \int_\square \psi_1 \varphi \,dy
+  \mu_1 \int_\square \frac{\partial \psi_1}{\partial x_1} \varphi \,dy.
\end{align}
Noting that $\partial_{x_1}\psi_1 \in H^1_0(\square)$, and using that $\mu_1, \psi_1$ is an eigenpair with test function
$\partial_{x_1}\psi_1$ in \eqref{eq:odin}, and the normalizing condition $\int_\square \psi_1^2 \,dy = 1$, yields after using $\psi_1$ as test function in \eqref{eq:dva},
\begin{align*}
\mu'_1 & = \int_\square \frac{\partial A}{\partial x_1}\nabla \psi_1 \cdot \nabla \psi_1  \,dy.\qedhere
\end{align*}
\end{proof}

Now we suppose that $A = A(x_1,y)$ is constant in $x_1$, and
let $\square(x_1)$ vary.
To compute $\mu_1'(x_1)$ we need to compute how points on the boundary of $\square(x_1)$ move in the normal
direction when $x_1$ is varied, i.e. the normal velocity of the boundary,
which we compute it terms of $F$.
By definition,
\begin{align*}
\square(x_1) = \{ y : F(x_1, y) > 0 \}.
\end{align*}
The boundary of $\square(x_1)$ is given by
\begin{align*}
\partial \square(x_1) = \{ y : F(x_1, y) = 0 \}.
\end{align*}
Let $V_n$ denote the outward normal velocity, that is
\begin{align*}
V_n & = - \frac{\partial_{x_1} F }{|\nabla_y F|} 
\end{align*}
Let $\phi \in H^1(\square(x_1))$ be such that $\int_{\square(x_1)} \phi \,dy = 0$ and
\begin{align*}
-\Delta_y \phi & = 0 \,\,\qquad\qquad \text{in }  \square(x_1) \\ 
\nabla_y \phi \cdot \nu & = - \frac{\partial_{x_1}F}{|\nabla_y F|} \quad \text{ on } \partial \square(x_1)
\end{align*}
If the compatibility condition is not satisfied, we set $V = 0$ in some interior ball,
or curve of positive measure.
Then $V = \nabla_y \phi$ is a globally defined velocity field such that $V_n = V \cdot \nu = -\frac{\partial_{x_1}F}{|\nabla_y F|}$
on $\partial \square(x_1)$,
and where $\nu = -\frac{\nabla_y F}{|\nabla_y F|}$ is the outward unit normal to $\square(x_1)$.

\begin{lemma}\label{lm:nonconstantcell}
Let $A(x_1,y)$ and $\partial \square(x_1)$ be sufficiently smooth.
Suppose that $A(x_1,y)$ is constant in $x_1$.
Let $\mu_1, \psi_1$ be the principle eigenpair to
\begin{align*}
-\mop{div}_y(Ay)\nabla_y \psi) & = \mu(x_1) \psi \quad \text{ in } \square(x_1),
\end{align*}
with homogeneous Dirichlet condition on $\partial \square(x_1)$, 
 normalized by
\begin{align*}
\int_{\square(x_1)} \psi^2_1 \,dy = 1.
\end{align*}
Then
\begin{align}\label{eq:mu1primex1}
\mu_1' & =
 2 \int_{\square(x_1)} ( A(y)\nabla_y \psi_1 \cdot \nabla_y (\nabla_y \psi_1 \cdot V )
- \mu_1 \psi_1 (\nabla_y \psi_1 \cdot V )   ) \,dy \\ 
& \quad - \int_{\partial \square(x_1)} (A(y) \nabla_y \psi_1 \cdot \nabla_y \psi_1) V_n  \,d\sigma. \notag
\end{align}
\end{lemma}
\begin{proof}
Let $\dot v = \partial_{x_1} v - \nabla_y v \cdot V$ denote the material derivative.
For any test function $\varphi \in H^1_0(\square)$,
\begin{align}\label{eq:odin4}
\int_{\square(x_1)}A(y) \nabla_y \psi_1\cdot \nabla_y \varphi \,dy & = \mu_1 \int_{\square(x_1)} \psi_1 \varphi \,dy.
\end{align}
Differentiate both sides with respect to $x_1$, to obtain
\begin{align}\label{eq:dva7}
& \int_{\square(x_1)} \Big( A(y)\nabla_y \frac{\partial \psi_1}{\partial x_1} \cdot \nabla_y \varphi
+
A(y)\nabla_y \psi_1 \cdot \nabla_y \frac{\partial \varphi}{\partial x_1} 
\Big) \,dy \notag\\
& \quad
- \int_{\partial \square(x_1)}
(A(y)\nabla_y \psi_1 \cdot \nabla_y \varphi)V_n
\,d\sigma 
+  \int_{\square(x_1)} A(y)\nabla_y \dot \psi_1 \cdot \nabla_y \varphi \,dy \notag\\
& =
 \mu_1' \int_{\square(x_1)} \psi_1 \varphi \,dy
 +
 \mu_1 \int_{\square(x_1)} \Big( \frac{\partial \psi_1}{\partial x_1} \varphi + \psi_1 \frac{\partial \varphi}{\partial x_1} \Big) \,dy \\
& \quad  - \mu_1 \int_{\partial \square(x_1)} \psi_1 \varphi V_n \,d\sigma 
  + \mu_1 \int_{\square(x_1)} \dot \psi_1 \varphi \,dy.\notag
\end{align}
Use $\psi_1$ as a test function in \eqref{eq:dva7},
substitute $\partial_{x_1} \psi_1 = \dot \psi_1 - \nabla_y \psi_1 \cdot V$,
and note that $\dot \psi_1$ can be used as a test function in \eqref{eq:odin4}.
Using that $\psi_1 = 0$ on $\partial \square(x_1)$, and the normalization $\int_{\square(x_1)} \psi_1^2 \,dy = 1$, yield \eqref{eq:mu1primex1}.
\end{proof}

\bibliographystyle{plain}
\bibliography{refs}

\end{document}